\documentclass[reqno, 12pt]{amsart}
\pdfoutput=1
\makeatletter
\let\origsection=\section \def\section{\@ifstar{\origsection*}{\mysection}}
\def\mysection{\@startsection{section}{1}\z@{.7\linespacing\@plus\linespacing}{.5\linespacing}{\normalfont\scshape\centering\S}}
\makeatother

\usepackage{amsmath,amssymb,amsthm}
\usepackage{mathrsfs}
\usepackage{mathabx}\changenotsign
\usepackage{dsfont}
\usepackage{afterpage}

\usepackage[normalem]{ulem}

\usepackage[dvipsnames]{xcolor}
\usepackage[backref]{hyperref}
\hypersetup{
    colorlinks,
    linkcolor={red!60!black},
    citecolor={green!60!black},
    urlcolor={blue!60!black}
}

\usepackage{graphicx}

\usepackage[open,openlevel=2,atend]{bookmark}

\usepackage[abbrev,msc-links,backrefs]{amsrefs}
\usepackage{doi}

\renewcommand{\PrintDOI}[1]{\doi{#1}}

\usepackage[T1]{fontenc}
\usepackage{lmodern}
\usepackage[babel]{microtype}
\usepackage[english]{babel}

\usepackage{geometry}
\numberwithin{equation}{section}
\numberwithin{figure}{section}

\usepackage{enumitem}

\usepackage{pgf,tikz}
\usetikzlibrary{math,decorations.pathreplacing,calligraphy}
\usepackage{float}
\usepackage{longtable}
\usepackage{makecell}
\usepackage{soul}

\AtBeginDocument{%
   \def\MR#1{}
}

\usepackage{hyperref}

\usepackage{caption}
\makeatletter
\def\greek#1{\expandafter\@greek\csname c@#1\endcsname}
\def\Greek#1{\expandafter\@Greek\csname c@#1\endcsname}
\def\@greek#1{\ifcase#1
	\or $\alpha$%
	\or $\beta$%
	\or $\gamma$%
	\or $\delta$%
	\or $\epsilon$%
	\or $\zeta$%
	\or $\eta$%
	\or $\theta$%
	\or $\iota$%
	\or $\kappa$%
	\or $\lambda$%
	\or $\mu$%
	\or $\nu$%
	\or $\xi$%
	\or $o$%
	\or $\pi$%
	\or $\rho$%
	\or $\sigma$%
	\or $\tau$%
	\or $\upsilon$%
	\or $\phi$%
	\or $\chi$%
	\or $\psi$%
	\or $\omega$%
\fi}
\def\@Greek#1{\ifcase#1
	\or $\mathrm{A}$%
	\or $\mathrm{B}$%
	\or $\Gamma$%
	\or $\Delta$%
	\or $\mathrm{E}$%
	\or $\mathrm{Z}$%
	\or $\mathrm{H}$%
	\or $\Theta$%
	\or $\mathrm{I}$%
	\or $\mathrm{K}$%
	\or $\Lambda$%
	\or $\mathrm{M}$%
	\or $\mathrm{N}$%
	\or $\Xi$%
	\or $\mathrm{O}$%
	\or $\Pi$%
	\or $\mathrm{P}$%
	\or $\Sigma$%
	\or $\mathrm{T}$%
	\or $\mathrm{Y}$%
	\or $\Phi$%
	\or $\mathrm{X}$%
	\or $\Psi$%
	\or $\Omega$%
\fi}

\AddEnumerateCounter{\greek}{\@greek}{24}
\AddEnumerateCounter{\Greek}{\@Greek}{12}

\makeatother

\let\polishlcross=\l
\def\l{\ifmmode\ell\else\polishlcross\fi}

\def\paragraph#1{%
  \noindent\textbf{#1.}\enspace}

\let\emptyset=\varnothing
\let\setminus=\smallsetminus

\makeatletter
\def\moverlay{\mathpalette\mov@rlay}
\def\mov@rlay#1#2{\leavevmode\vtop{   \baselineskip\z@skip \lineskiplimit-\maxdimen
   \ialign{\hfil$\m@th#1##$\hfil\cr#2\crcr}}}
\newcommand{\charfusion}[3][\mathord]{
    #1{\ifx#1\mathop\vphantom{#2}\fi
        \mathpalette\mov@rlay{#2\cr#3}
      }
    \ifx#1\mathop\expandafter\displaylimits\fi}
\makeatother

\DeclareFontFamily{U}  {MnSymbolC}{}
\DeclareSymbolFont{MnSyC}         {U}  {MnSymbolC}{m}{n}
\DeclareFontShape{U}{MnSymbolC}{m}{n}{
    <-6>  MnSymbolC5
   <6-7>  MnSymbolC6
   <7-8>  MnSymbolC7
   <8-9>  MnSymbolC8
   <9-10> MnSymbolC9
  <10-12> MnSymbolC10
  <12->   MnSymbolC12}{}
\DeclareMathSymbol{\powerset}{\mathord}{MnSyC}{180}

\usepackage{tikz}
\usetikzlibrary{calc,decorations.pathmorphing}
\pgfdeclarelayer{background}
\pgfdeclarelayer{foreground}
\pgfdeclarelayer{front}
\pgfsetlayers{background,main,foreground,front}

\let\epsilon=\varepsilon
\let\eps=\epsilon
\let\rho=\varrho
\let\theta=\vartheta
\let\kappa=\varkappa

\let\E=\EE

\def\PP{\Pr}

\def\rm{\mathbb{RM}}

\newcommand{\cC}{\mathcal{C}}

\newcommand{\cF}{\mathcal{F}}

\newcommand{\cM}{\mathcal{M}}

\newcommand{\cP}{\mathcal{P}}
\newcommand{\cQ}{\mathcal{Q}}
\newcommand{\cR}{\mathcal{R}}

\theoremstyle{plain}
\newtheorem{thm}{Theorem}[section]
\newtheorem{theorem}[thm]{Theorem}

\newtheorem{prop}[thm]{Proposition}

\newtheorem{cor}[thm]{Corollary}
\newtheorem{c-cor}[thm]{Co-Corollary}
\newtheorem{lemma}[thm]{Lemma}

\newtheorem{observation}[thm]{Observation}

\theoremstyle{definition}
\newtheorem{rem}[thm]{Remark}

\newtheorem*{dfn*}{Definition}
\newtheorem{exmp}{Example}

\usepackage{accents}

\let\phi=\varphi

\DeclareMathOperator{\tr}{tr}

\begin{document}

\title[Substructures in random ordered matchings]{Homogeneous substructures in random ordered hyper-matchings}

\author{Andrzej Dudek}
\address{Department of Mathematics, Western Michigan University, Kalamazoo, MI, USA}
\email{\tt andrzej.dudek@wmich.edu}
\thanks{The first author was supported in part by Simons Foundation Grant MPS-TSM-00007551.}

\author{Jaros\l aw Grytczuk}
\address{Faculty of Mathematics and Information Science, Warsaw University of Technology, Warsaw, Poland}
\email{jaroslaw.grytczuk@pw.edu.pl}
\thanks{The second author was supported in part by Narodowe Centrum Nauki, grant 2020/37/B/ST1/03298.}

\author{Jakub Przyby{\l}o}
\address{AGH University of Krakow, al. A. Mickiewicza 30, 30-059 Krakow, Poland}
\email{jakubprz@agh.edu.pl}
\thanks{The third author was supported in part by the AGH University of Krakow grant no. 16.16.420.054, funded by the Polish Ministry of Science and Higher Education.}

\author{Andrzej Ruci\'nski}
\address{Department of Discrete Mathematics, Adam Mickiewicz University, Pozna\'n, Poland}
\email{\tt rucinski@amu.edu.pl}
\thanks{The last author was supported by Narodowe Centrum Nauki,  grant 2024/53/B/ST1/00164.}

\begin{abstract}An ordered $r$-uniform matching of size $n$ is a collection of $n$ pairwise disjoint $r$-subsets of a linearly ordered set of $rn$ vertices. For $n=2$, such a matching is called an \emph{$r$-pattern}, as it represents one of $\tfrac12\binom{2r}r$ ways two disjoint edges may intertwine. Given a  set $\cP$ of $r$-patterns, a \emph{$\cP$-clique} is a matching with all pairs of edges order-isomorphic to a member of $\cP$.

In this paper we are interested in the size of a largest $\cP$-clique in a \emph{random} ordered $r$-uniform matching selected uniformly from all such matchings on a fixed vertex set $[rn]$.
	We determine this size (up to multiplicative constants) for several sets $\cP$, including all sets of size $|\cP|\le2$, the set $\cR^{(r)}$ of all $r$-partite patterns, as well as  sets $\cP$ enjoying a Boolean-like, symmetric structure.
\end{abstract}

\maketitle



\section{Introduction}
In this paper we are interested in the maximum size of sub-matchings of random ordered uniform matchings with all pairs of edges forming prescribed patterns. Before stating our results we present necessary definitions and background.

\subsection{Matchings and words}
A hypergraph on a linearly ordered vertex set is called \emph{ordered},  and it is \emph{$r$-uniform}, $r\ge1$, if all its edges have the same size $r$.
Throughout the paper, $r$ is always a fixed constant, while  $n$  will grow to infinity.

Two ordered hypergraphs are \emph{order-isomorphic} if there is an isomorphism between them preserving their linear orders.
An \emph{ordered $r$-matching} of \emph{size} $n$ is an ordered $r$-uniform hypergraph consisting of $n$ pairwise disjoint edges (and no isolated vertices). For fixed $r$ and $n$, there are precisely
$\frac{(rn)!}{(r!)^n\, n!}$ distinct ordered $r$-matchings of size $n$. The family of all of them will be denoted by $\cM^{(r)}_n$. For instance, the family $\cM^{(2)}_2$ consists of three distinct matchings depicted in Figure \ref{ESZ1}.

\begin{figure}[ht]
	\captionsetup[subfigure]{labelformat=empty}
	\begin{center}
		
		\scalebox{1}
		{
			\centering
			\begin{tikzpicture}
				[line width = .5pt,
				vtx/.style={circle,draw,black,very thick,fill=black, line width = 1pt, inner sep=0pt},
				]
				
				\coordinate (0) at (0.5,0) {};
				\node[vtx] (1) at (1,0) {};
				\node[vtx] (2) at (2,0) {};
				\node[vtx] (3) at (3,0) {};
				\node[vtx] (4) at (4,0) {};
				\coordinate (5) at (4.5,0) {};
				\draw[line width=0.3mm, color=lightgray]  (0) -- (5);
				\fill[fill=black, outer sep=1mm]  (1) circle (0.1) node [below] {$1$};
				\fill[fill=black, outer sep=1mm]  (2) circle (0.1) node [below] {$2$};
				\fill[fill=black, outer sep=1mm]  (3) circle (0.1) node [below] {$3$};
				\fill[fill=black, outer sep=1mm]  (4) circle (0.1) node [below] {$4$};
				\draw[line width=0.5mm, color=black, outer sep=2mm] (2) arc (0:180:0.5);
				\draw[line width=0.5mm, color=black, outer sep=2mm] (4) arc (0:180:0.5);
				
				\coordinate (0) at (5.5,0) {};
				\node[vtx] (1) at (6,0) {};
				\node[vtx] (2) at (7,0) {};
				\node[vtx] (3) at (8,0) {};
				\node[vtx] (4) at (9,0) {};
				\coordinate (5) at (9.5,0) {};
				\draw[line width=0.3mm, color=lightgray]  (0) -- (5);
				\fill[fill=black, outer sep=1mm]  (1) circle (0.1) node [below] {$1$};
				\fill[fill=black, outer sep=1mm]  (2) circle (0.1) node [below] {$2$};
				\fill[fill=black, outer sep=1mm]  (3) circle (0.1) node [below] {$3$};
				\fill[fill=black, outer sep=1mm]  (4) circle (0.1) node [below] {$4$};
				\draw[line width=0.5mm, color=black, outer sep=2mm] (3) arc (0:180:0.5);
				\draw[line width=0.5mm, color=black, outer sep=2mm] (4) arc (0:180:1.5);
				
				\coordinate (0) at (10.5,0) {};
				\node[vtx] (1) at (11,0) {};
				\node[vtx] (2) at (12,0) {};
				\node[vtx] (3) at (13,0) {};
				\node[vtx] (4) at (14,0) {};
				\coordinate (5) at (14.5,0) {};
				\draw[line width=0.3mm, color=lightgray]  (0) -- (5);
				\fill[fill=black, outer sep=1mm]  (1) circle (0.1) node [below] {$1$};
				\fill[fill=black, outer sep=1mm]  (2) circle (0.1) node [below] {$2$};
				\fill[fill=black, outer sep=1mm]  (3) circle (0.1) node [below] {$3$};
				\fill[fill=black, outer sep=1mm]  (4) circle (0.1) node [below] {$4$};
				\draw[line width=0.5mm, color=black, outer sep=2mm] (3) arc (0:180:1);
				\draw[line width=0.5mm, color=black, outer sep=2mm] (4) arc (0:180:1);
				
			\end{tikzpicture}
		}
		
	\end{center}
	\caption{Three distinct $2$-matchings of size two.}
	\label{ESZ1}
	
\end{figure}

A natural and convenient way to represent ordered $r$-matchings is in terms of \emph{words}. One simply fixes an alphabet of size $n$ and assigns different letters to different edges (so the obtained word contains each letter exactly $r$ times). For instance, if $r=3$ and $M=\{\{1,2,4\}, \{3,10,12\},\{5,7,11\},\{6,8,9\}\}$,
then two examples of words representing $M$ are $AABACDCDDBCB$ and $CCECDFDFFEDE$. Obviously, there are many such words which differ only in the letters selected to the alphabet and their assignment to the edges. All such words are  equivalent and throughout the paper we identify ordered matchings with the equivalence classes of their  word representations. Just for convenience, to represent a given matching we typically choose a word in which the first occurrences of letters appear alphabetically.

 An important subclass of matchings can be defined via
 the notion of interval chromatic number introduced in~\cite{Pach-Tardos}. An \emph{interval coloring} of an ordered hypergraph $G$ is a partition of its linearly ordered vertex set into intervals (colors) such that each edge of $G$ has at most one vertex in each of these intervals. The least number of colors in any interval coloring of $G$ is the \emph{interval chromatic number} of $G$, denoted by $\chi_{<}(G)$.
Clearly, if $G$ is an ordered $r$-uniform hypergraph, then  $\chi_{<}(G)\ge r$. An ordered $r$-matching $M$ with $\chi_<(M)=r$ will be called \emph{$r$-partite}. In such a matching, every edge contains exactly one vertex from each of the $r$ consecutive blocks of size $n$. Equivalently, any word representing an  $r$-partite $r$-matching must consist of $r$ consecutive blocks each containing all letters. Note also that, trivially, every sub-matching of an $r$-partite $r$-matching is $r$-partite itself.

For instance, the following $3$-matching is visibly $3$-partite:
\[{\color{red}{ABCDEFG}}\:\!{\color{cyan}{\underline{ACGEFBD}}}\:\!{\color{blue}{\underline{\underline{GFEDCBA}}}}.\]
Note that here every edge is of the form $\color{red}{X}\:\!\color{cyan}{\underline{X}}\:\!\color{blue}{\underline{\underline{X}}}$.
Also, the sub-matching consisting only of the edges represented by letters $A$ and $B$, $ABABBA$, is 3-partite too.

We will often use the power notation  $W^k$ to write concisely the word
$$\underbrace{W\;W\;\cdots W}_{k}$$ consisting of $k$ repetitions of the same word $W$. For instance, $ABABAB=(AB)^3$ and $ABBBABAAAA=AB^3ABA^4$.

\subsection{Patterns and cliques}
Ordered $r$-matchings of size $n=2$ are called \emph{$r$-patterns}. For each $r\geqslant 2$ there are exactly $\tfrac12\binom{2r}r$ of them. The set of all $r$-patterns will be denoted  by $\cP^{(r)}$ (instead of $\cM_2^{(r)}$). For instance, in letter notation, $\cP^{(2)}=\{AABB,\;ABBA,\;ABAB\}$, (cf.  Figure \ref{ESZ1}).

For a fixed pattern $P$, a \emph{$P$-clique} is an ordered matching in which every pair of edges forms pattern $P$. In Figure \ref{ESZ2} one can see three $P$-cliques of size four corresponding to the three $2$-patterns in $\cP^{(2)}$.

\begin{figure}[ht]
	\captionsetup[subfigure]{labelformat=empty}
	\begin{center}
		
		\scalebox{1}
		{
			\centering
			\begin{tikzpicture}
				[line width = .5pt,
				vtx/.style={circle,draw,black,very thick,fill=black, line width = 1pt, inner sep=0pt},
				]
				
				\coordinate (0) at (0.25,0) {};
				\node[vtx] (1) at (0.5,0) {};
				\node[vtx] (2) at (1,0) {};
				\node[vtx] (3) at (1.5,0) {};
				\node[vtx] (4) at (2,0) {};
				\node[vtx] (5) at (2.5,0) {};
				\node[vtx] (6) at (3,0) {};
				\node[vtx] (7) at (3.5,0) {};
				\node[vtx] (8) at (4,0) {};
				\coordinate (9) at (4.25,0) {};
				\draw[line width=0.3mm, color=lightgray]  (0) -- (9);
				\fill[fill=black, outer sep=1mm]  (1) circle (0.1) node [below] {$A$};
				\fill[fill=black, outer sep=1mm]  (2) circle (0.1) node [below] {$A$};
				\fill[fill=black, outer sep=1mm]  (3) circle (0.1) node [below] {$B$};
				\fill[fill=black, outer sep=1mm]  (4) circle (0.1) node [below] {$B$};
				\fill[fill=black, outer sep=1mm]  (5) circle (0.1) node [below] {$C$};
				\fill[fill=black, outer sep=1mm]  (6) circle (0.1) node [below] {$C$};
				\fill[fill=black, outer sep=1mm]  (7) circle (0.1) node [below] {$D$};
				\fill[fill=black, outer sep=1mm]  (8) circle (0.1) node [below] {$D$};
				\draw[line width=0.5mm, color=black, outer sep=2mm] (2) arc (0:180:0.25);
				\draw[line width=0.5mm, color=black, outer sep=2mm] (4) arc (0:180:0.25);
				\draw[line width=0.5mm, color=black, outer sep=2mm] (6) arc (0:180:0.25);
				\draw[line width=0.5mm, color=black, outer sep=2mm] (8) arc (0:180:0.25);
				
				\coordinate (0) at (5.25,0) {};
				\node[vtx] (1) at (5.5,0) {};
				\node[vtx] (2) at (6,0) {};
				\node[vtx] (3) at (6.5,0) {};
				\node[vtx] (4) at (7,0) {};
				\node[vtx] (5) at (7.5,0) {};
				\node[vtx] (6) at (8,0) {};
				\node[vtx] (7) at (8.5,0) {};
				\node[vtx] (8) at (9,0) {};
				\coordinate (9) at (9.25,0) {};
				\draw[line width=0.3mm, color=lightgray]  (0) -- (9);
				\fill[fill=black, outer sep=1mm]  (1) circle (0.1) node [below] {$A$};
				\fill[fill=black, outer sep=1mm]  (2) circle (0.1) node [below] {$B$};
				\fill[fill=black, outer sep=1mm]  (3) circle (0.1) node [below] {$C$};
				\fill[fill=black, outer sep=1mm]  (4) circle (0.1) node [below] {$D$};
				\fill[fill=black, outer sep=1mm]  (5) circle (0.1) node [below] {$D$};
				\fill[fill=black, outer sep=1mm]  (6) circle (0.1) node [below] {$C$};
				\fill[fill=black, outer sep=1mm]  (7) circle (0.1) node [below] {$B$};
				\fill[fill=black, outer sep=1mm]  (8) circle (0.1) node [below] {$A$};
				\draw[line width=0.5mm, color=black, outer sep=2mm] (5) arc (0:180:0.25);
				\draw[line width=0.5mm, color=black, outer sep=2mm] (6) arc (0:180:0.75);
				\draw[line width=0.5mm, color=black, outer sep=2mm] (7) arc (0:180:1.25);
				\draw[line width=0.5mm, color=black, outer sep=2mm] (8) arc (0:180:1.75);
				
				\coordinate (0) at (10.25,0) {};
				\node[vtx] (1) at (10.5,0) {};
				\node[vtx] (2) at (11,0) {};
				\node[vtx] (3) at (11.5,0) {};
				\node[vtx] (4) at (12,0) {};
				\node[vtx] (5) at (12.5,0) {};
				\node[vtx] (6) at (13,0) {};
				\node[vtx] (7) at (13.5,0) {};
				\node[vtx] (8) at (14,0) {};
				\coordinate (9) at (14.25,0) {};
				\draw[line width=0.3mm, color=lightgray]  (0) -- (9);
				\fill[fill=black, outer sep=1mm]  (1) circle (0.1) node [below] {$A$};
				\fill[fill=black, outer sep=1mm]  (2) circle (0.1) node [below] {$B$};
				\fill[fill=black, outer sep=1mm]  (3) circle (0.1) node [below] {$C$};
				\fill[fill=black, outer sep=1mm]  (4) circle (0.1) node [below] {$D$};
				\fill[fill=black, outer sep=1mm]  (5) circle (0.1) node [below] {$A$};
				\fill[fill=black, outer sep=1mm]  (6) circle (0.1) node [below] {$B$};
				\fill[fill=black, outer sep=1mm]  (7) circle (0.1) node [below] {$C$};
				\fill[fill=black, outer sep=1mm]  (8) circle (0.1) node [below] {$D$};
				\draw[line width=0.5mm, color=black, outer sep=2mm] (5) arc (0:180:1);
				\draw[line width=0.5mm, color=black, outer sep=2mm] (6) arc (0:180:1);
				\draw[line width=0.5mm, color=black, outer sep=2mm] (7) arc (0:180:1);
				\draw[line width=0.5mm, color=black, outer sep=2mm] (8) arc (0:180:1);
				
			\end{tikzpicture}
		}
	
	\end{center}
	
	\caption{Cliques of size $4$ corresponding to the three patterns in $\cP^{(2)}$.}
	\label{ESZ2}
	
\end{figure}

It was proved in \cite{DGR-match} that every ordered $2$-matching of size $n$ contains, for \emph{some} $P\in \cP^{(2)}$, a $P$-clique of size $\lceil n^{1/3}\rceil$ (and this is optimal), while a \emph{random} such matching contains (with high probability) a $P$-clique of size $\Theta(n^{1/2})$ for \emph{every} $P\in \cP^{(2)}$. These results were generalized in \cite{JCTB_paper} to arbitrary $r\geqslant 3$, with a necessary restriction to \emph{collectable} patterns $P$, that is, to those patterns that allow building arbitrarily large $P$-cliques.

The following characterization of collectable patterns was established in \cite[Proposition 2.1]{JCTB_paper}: an $r$-pattern $P$ is collectable if and only if it is \emph{splittable} into consecutive \emph{blocks}, $P=S_1\cdots S_s$, of the form $S_i=A^tB^t$ or $S_i=B^tA^t$. Obviously, such a splitting is unique.  For instance,
 \begin{equation}\label{P}
 P=AB|BBAA|BA|AAABBB|BBAA|BA|AB|AB
 \end{equation}
is a collectable $12$-pattern with the suitable splitting into eight blocks $S_i$.
 From this characterization one can easily deduce that for every $r\geqslant 3$ there are exactly $3^{r-1}$ collectable patterns in $\cP^{(r)}$. It was also proved in \cite[Proposition 2.1]{JCTB_paper} that non-collectable patterns $P$ cannot even create $P$-cliques of size three.

 Every collectable $r$-pattern $P$ uniquely determines a \emph{composition} $\lambda_P=(\lambda_1,\dots,\lambda_s)$ of the integer $r$, that is, an \emph{ordered partition} $r=\lambda_1+\cdots+\lambda_s$, where $\lambda_i=|S_i|/2$, for $i=1,\dots,s$.
  For the 12-pattern $P$ defined in~\eqref{P} we have $\lambda_P=(1,2,1,3,2,1,1,1)$.

Note that for a collectable pattern $P$ with $s$ blocks in the splitting there are exactly $2^{s-1}$  patterns $Q$ with $\lambda_Q=\lambda_P$ (including $P$).   Let us call such pairs $\{P,Q\}$  \emph{harmonious}.
For an integer $r\ge2$ and a composition $\lambda$ of $r$, let $\mathcal P(\lambda)$ be the set of all $r$-patterns $P$ with $\lambda_P=\lambda$. So, all members of $\mathcal P(\lambda)$ are mutually harmonious.

Recall that any word representing an  $r$-partite $r$-matching must consist of $r$ consecutive blocks each containing all letters. Thus, $r$-partite $r$-patterns can be characterized as those which take on the form $P=S_1\cdots S_r$, where $S_1=AB$, while any other block $S_i$ is either $AB$ or $BA$.
   Let $\lambda^{(r)}:=(1,1,\dots,1)$. Then $\cR^{(r)}:=\cP(\lambda^{(r)})$ is the set of all $r$-partite $r$-patterns and $|\cR^{(r)}|=2^{r-1}$ for all $r\geqslant2$.


Table \ref{table:relations} contains all patterns in $\cP^{(3)}$ among which $P_1,P_2,\ldots,P_9$ are collectable, while $P_0$ is the only pathological one. Moreover,  $\cR^{(3)}=\{P_6,P_7, P_8,P_9\}$. Clearly, all members of $\cR^{(3)}$ are mutually harmonious, so are
$P_2=AABBBA$ and $P_3=AABBAB$ (with composition $\lambda=(2,1)$) as well as $P_4=ABBBAA$ and $P_5=ABAABB$ (with composition $\lambda=(1,2)$).

It was proved in \cite{JCTB_paper} that for $r\ge2$ \emph{every} ordered $r$-matching of size $n$ contains, for \emph{some} collectable pattern $P\in \cP^{(r)}$, a $P$-clique of size  $\Omega(n^{1/3^{r-1}})$, while a \emph{random} ordered $r$-matching contains a $P$-clique of size $\Theta(n^{1/r})$ for \emph{every} collectable pattern $P\in \cP^{(r)}$. The former result was subsequently improved to $\Theta(n^{1/(2^{r}-1)})$ by combined work of Anastos, Jin, Kwan, and Sudakov \cite{AJKS} and Sauermann and Zakharov \cite{SZ}.

\begin{table}[h!]
	\centering
	\begin{tabular}{ |c|c| }
		\hline
		$P_0=AABABB$&$P_5=ABAABB$\\
		\hline
		$P_1=AAABBB$&$P_6=ABBABA$\\
		\hline
		$P_2=AABBBA$&$P_7=ABBAAB$\\
		\hline
		$P_3=AABBAB$&$P_8=ABABBA$\\
		\hline
		$P_4=ABBBAA,$&$P_9=ABABAB$\\
			\hline
	\end{tabular}
\caption{All 3-patterns; among them only $P_0$ is non-collectable, while $P_6-P_9$ are 3-partite.}
\label{table:relations}
\end{table}

\subsection{Multi-patterned  cliques}
In this paper our goal is to study more complex substructures of random $r$-matchings than $P$-cliques. Let $\cP \subseteq \cP^{(r)}$ be a fixed subset of $r$-patterns. An ordered $r$-matching $M$ is called a \emph{$\cP$-clique} if every pair of edges in $M$ forms a pattern belonging to~$\cP$. Let $z_{\cP}(M)$ denote the maximum size of a $\cP$-clique in $M$. Note that if $\cQ\subset \cP$, then every $\cQ$-clique is a $\cP$-clique and thus, $z_{\cQ}(M)\le z_{\cP}(M)$. Also, if $N\subset M$, then $z_{\cP}(N)\le z_{\cP}(M)$. Throughout, we will be referring to both these trivial properties as \emph{monotonicity}.


 Let $\rm^{(r)}_{n}$ denote a random ordered $r$-matching of size $n$, that is, an $r$-matching picked uniformly at random from the set ${\mathcal M}_n^{(r)}$ of all ordered $r$-matchings on the set $[rn]$. Using the introduced notation, our main problem can be stated as follows.

 {\bf Problem.} 
 	\emph{For a fixed $r\geqslant 2$ and for a given subset of patterns $\cP\subseteq \cP^{(r)}$, determine the order of magnitude of the random variable $z_{\cP}(\rm^{(r)}_{n})$.}

By  ``order of magnitude'' we mean  some function $f(n)$ such that $z_{\cP}(\rm_n^{(r)})=\Theta(f(n))$ \emph{asymptotically almost surely}, that is, with probability  tending to one as $n\to\infty$. In the course of the paper we will use the common abbreviation, namely \emph{a.a.s.}, for this phrase. Occasionally, we will be able to pinpoint the value of $z_{\cP}(\rm_n^{(r)})$ \emph{asymptotically}. We write $a_n\sim b_n$ whenever $a_n/b_n\to1$ as $n\to\infty$ and, for a sequence of random variables $X_n$, we write $X_n\sim a_n$ if $X_n/a_n$ converges to 1 in probability.

 As was signaled earlier, the order of magnitude of $z_{\cP}(\rm^{(r)}_{n})$ when $|\cP|=1$ has already been determined.
		
\begin{thm}[\cite{JCTB_paper}]\label{thm:random_one_pattern} For $r\ge2$ and every collectable $r$-pattern $P$, a.a.s.,
	\[
	z_{P}(\rm^{(r)}_{n})=\Theta(n^{1/r}).
	\]
\end{thm}

Let us mention that for $r=2$ the above result had been proved, even in a sharper form, already by Baik and Rains \cite{BaikRains} for $P=ABBA$ and $P=ABAB$, and by Justicz, Scheinerman, and Winkler \cite{JSW} for $P=AABB$.

The only other case that has been investigated so far was $r=2$ and $|\cP|=2$.
It is natural to expect that the value of $z_{\cP}(\rm^{(r)}_{n})$ depends on the size $|\cP|$. However, as we will soon see, the situation is far more intricate. 
There are three possible pairs of $2$-patterns:  $\cP_1=\{AABB,ABAB\}$, $\cP_2=\{AABB, ABBA\}$ and $\cR=\{ABAB,ABBA\}$. It turns out that typically the largest size of an $\cR$-clique  differs significantly from the largest sizes of $\cP_i$-cliques, $i=1,2$.
\begin{thm}[\cite{DGR-socks}]\label{thm:random_two_patterns_r=2}
	Let $\cP_1$, $\cP_2$, and $\cR$ be as defined above. Then, a.a.s.,
	\[z_{\cP_i}(\rm^{(2)}_{n})=\Theta(n^{1/2}),\] for $i=1,2$, while
	\[z_{\cR}(\rm^{(2)}_{n})\sim \frac n2.\]
\end{thm}

The first formula was only mentioned without proof in the Final Remarks section of \cite{DGR-socks}; therefore we include a simple proof here (see Remark~\ref{r2} below). The second formula follows from Theorem 1.2 in \cite{DGR-socks} and  was also obtained  (much) earlier by Scheinerman \cite{Scheinerman1988}, via a different approach.

The reason behind such a diverse behavior of  random variables $z_{\cP}(\rm^{(2)}_{n})$ with the same size of $\cP$ is structural. Indeed, the pairs of patterns appearing in $\cP_1$ and $\cP_2$ are not harmonious, unlike the patterns in $\cR$. In fact, $\cR=\cR^{(2)}$, which leads to a reformulation of the second statement in Theorem~\ref{thm:random_two_patterns_r=2}:  \emph{The maximum size of a $2$-partite sub-matching in $\rm^{(2)}_{n}$ is asymptotically equal to $n/2$}. This reformulation is legitimate, because, more generally, it is easy to see that for every $r\ge2$, a matching is $r$-partite if and only if it is an $\cR^{(r)}$-clique.

 Theorems~\ref{thm:random_one_pattern} (for $r=2$) and~\ref{thm:random_two_patterns_r=2} can be stated in terms of graphs defined on randomly selected intervals.
 Indeed, any ordered matching $M$ of size $n$ can be interpreted as a set of $n$ intervals (identified with the edges of $M$) with no common endpoints. For a subset $\cP\subset\cP^{(2)}$ we may define a graph $G(M,\cP)$ on vertex set $V(G)=M$, where for any $e,f\in M$, we have $\{e,f\}\in E(G)$ if $e$ and $f$ form a pattern belonging to $\cP$. The six non-trivial subsets of $\cP^{(2)}$ give rise to six well studied types of graphs, among them interval graphs (for $\cP=\cR$). For instance, Theorem~\ref{thm:random_two_patterns_r=2}, with $\cP=\cR$, estimates  the clique number of a random interval graph of size $n$ defined by $\rm^{(2)}_{n}$, which was the actual motivation for  Scheinerman to study random matchings in \cite{Scheinerman1988}.

\section{New results}
Theorems~\ref{thm:random_one_pattern} and~\ref{thm:random_two_patterns_r=2} quoted in Introduction give us a  full picture of the order of magnitude of $z_{\cP}(\rm^{(2)}_{n})$, the size of the largest $\cP$-clique in $\rm^{(2)}_{n}$, for $r=2$.
For $r\ge3$ our grasp on $z_{\mathcal P}(\rm_n)$ is still far from complete. In this section we present  our new results.

\subsection{Generalizations of Theorem~\ref{thm:random_two_patterns_r=2}}

One case in which we are able to pinpoint the value of $z_{\mathcal P}(\rm^{(r)}_n)$ quite precisely is when $\cP$ consists of \emph{all} $r$-partite $r$-patterns, that is, when $\cP =\cR^{(r)}$. Generalizing the second part of Theorem~\ref{thm:random_two_patterns_r=2}, we show that the largest $\cR^{(r)}$-clique in $\rm^{(r)}_{n}$ has size linear in $n$ and find the value of the multiplicative constant asymptotically.

\begin{thm}\label{thm:random_r-partite}
Let $r\ge2$, and let $\cR^{(r)}$ be the set of all $r$-partite $r$-patterns. Then, a.a.s., \[z_{\cR^{(r)}}(\rm^{(r)}_{n})\sim\frac{(r-1)!}{r^{r-1}}\cdot n.\]
\end{thm}

As noted before, any $r$-partite sub-matching of $\cM_n^{(r)}$ is an $\cR^{(r)}$-clique, hence the same asymptotic formula holds for the maximum size of such a sub-matching in $\rm^{(r)}_{n}$. The following statement is a convenient reformulation of Theorem~\ref{thm:random_r-partite}.

\begin{cor}\label{Mpart} Let $r\ge2$, and let $R_n^{(r)}$ be the largest (lexicographically first) $r$-partite sub-matching of $\rm^{(r)}_{n}$. Then, a.a.s.,
\[
|R_n^{(r)}|\sim\frac{(r-1)!}{r^{r-1}} \cdot n.
\]
\end{cor}

Another class of patterns for which we can determine the order of magnitude of $z_{\mathcal P}(\rm^{(r)}_{n})$ for all $r\ge2$ is that of \emph{all} sets of $r$-patterns $\mathcal{P}$ with $|\cP|=2$. Generalizing the first part of Theorem~\ref{thm:random_two_patterns_r=2}, we are able to prove the following result.
  We call a pair of  $r$-patterns $\{P,Q\}$ a \emph{mismatch} if one of them is collectable, while the other is not, or both are collectable, but then not harmonious. For instance, for $r=3$, the pairs $P_0=AABABB, P_1=AAABBB$ and $P_1=AAABBB, P_2=AABBBA$ are both mismatches. When $\{P,Q\}$ is a mismatch, we will also say that $P$ and $Q$ are \emph{mismatched}.

\begin{thm}\label{thm:random_two_patterns_gen}
For $r\ge2$, let $\cP=\{P,Q\}$ be a pair of distinct $r$-patterns. Then, a.a.s.,
\[
z_{\mathcal P}(\rm^{(r)}_{n})=\begin{cases}\Theta(n^{1/(r-1)})&\text{if $P$ and $Q$ are harmonious,}\\
	\Theta(n^{1/r})&\text{if $P$ and $Q$ are mismatched,}\\
z&\text{if neither $P$ nor $Q$ is collectable,}
\end{cases}
\]
where $z=z_{\cP}$ is an integer, $2\le z\le 5$.
\end{thm}
\noindent It is worth noticing that the thresholds for $z_{\mathcal P}(\rm^{(r)}_{n})$ in the above theorem do not depend on $\lambda_P$ or $\lambda_Q$ but only on $r$ and the relation between $P$ and $Q$. As we will demonstrate in the proof, the constant $z$ appearing in the last statement equals the largest size of a $\{P,Q\}$-clique whatsoever.

\medskip

An application of Theorem~\ref{thm:random_two_patterns_gen} (the mismatch case) to families of axis-parallel random rectangles obtained from $\rm^{(4)}_{n}$ has been described in~\cite{DGPR-LATIN} (see Section 1.2 and the comment after Theorem 4 therein).


\subsection{Boolean cubes}\label{bool}

The most challenging part of Theorem~\ref{thm:random_two_patterns_gen} concerns harmonious pairs. The approach we are using in that case, however, is applicable in a far more general setting of
arbitrarily large yet highly symmetric sets of patterns.
In order to describe them, let us consider a collectable pattern $P$
with $\lambda_P=(\lambda_1,\dots,\lambda_s)$ and its splitting $P=S_1\cdots S_s$ into blocks of sizes $|S_i|=2\lambda_i$. Next, let a partition $\pi$ of the form $[s]=T_0\cup T_1\cup\cdots \cup T_t$ be given, with $1\in T_0$ and no $T_j=\emptyset$.
 For $j=0,\dots,t$, set $P_j$ for the concatenation of segments $S_i$, with $i\in T_j$. Let us call these $P_j$'s the \emph{mega-blocks} of~$P$.

 A \emph{flip} of a mega-block involves exchanging all letters $A$ for $B$ and vice versa, $B$ for $A$, in this mega-block. For a subset $F\subseteq \{1,\dots,t\}$, let $P_F$ be the pattern obtained from $P$ by flipping all mega-blocks $P_j$ with $j\in F$. Clearly, the patterns $P$ and $P_F$ are harmonious.
 For instance, for  $P$ as in \eqref{P} and $\pi$ given by $[8]=\{1,3,5\}\cup\{2,8\}\cup \{4,6,7\}$, the pattern $Q=AB|\color{blue}{AABB}\color{black}{|BA|AAABBB|BBAA|BA|AB|}\color{blue}BA$ is obtained by flipping the mega-block $P_1=S_2S_8$, so we have $Q=P_{\{1\}}$.

 \begin{dfn*}
 Any set of $r$-patterns of the form
 $$\cC(P,\pi)=\{P_F:\;\;F\subseteq[t]\}$$
 is called a \emph{$t$-cube} or just \emph{cube} if the dimension is not specified.
 \end{dfn*}

  To explain this terminology note that $\cC(P,\pi)$  corresponds one-to-one to the Boolean cube $\{A,B\}^t$, where the letter on the $j$-th position means that the $j$-th mega-block begins with that letter. For practical purposes we distinguish one pattern ($P$) as a sort of creator of the cube, though, due to symmetry, every pattern of the cube could play that role.
 Clearly, $|\cC(P,\pi)|=2^t$.

\begin{exmp}\label{12} The following set of four 12-patterns
\begin{eqnarray}
P_{\emptyset}=AB|BBAA|BA|AAABBB|BBAA|BA|AB|AB \label{1}\\
P_{\{1\}}=AB|\color{blue}{AABB}\color{black}{|BA|AAABBB|BBAA|BA|AB|}\color{blue}{BA} \label{2}\\
P_{\{2\}}=AB|BBAA|BA|\color{blue}{BBBAAA}\color{black}{|BBAA|}\color{blue}{AB}\color{black}{|}\color{blue}{BA}\color{black}|AB \label{3}\\
P_{\{1,2\}}=AB|\color{blue}{AABB}\color{black}{|BA|}\color{blue}{BBBAAA}\color{black}{|BBAA|}\color{blue}{AB}\color{black}{|}\color{blue}{BA}
\color{black}{|}\color{blue}{BA} \label{4}
\end{eqnarray}
forms a 2-cube $\cC(P,\pi)$ generated by the pattern $P=P_{\emptyset}$, defined also in \eqref{P}, and the partition $\pi$ given by $[8]=\{1,3,5\}\cup\{2,8\}\cup \{4,6,7\}$.
 \end{exmp}
Notice that for any collectable pattern $P$, the set $\cP(\lambda_P)$ of all patterns harmonious with $P$ is an $(s-1)$-cube, where $s$ is the length of composition $\lambda_P$. Indeed, $\cP(\lambda_P)=\cC(P,\pi)$ with  $\pi$ being  the partition of $[s]$ into singletons. At the other extreme,  observe that a 0-cube is just the singleton set $\{P\}$, while  if $P$ and $Q$ form a harmonious pair, then they form a 1-cube ${\mathcal P}_P(\pi)$, where $\pi$ is the partition of $[s]$ into just two sets, $T_0$ and $T_1$, with $T_1$ corresponding to the flip transforming $P$ into $Q$. For example, if $P=ABBAABAB$ and $Q=ABBABABA$, then $s=4$, $T_0=\{1,2\}$,  $T_1=\{3,4\}$, and so $\{P,Q\}=\cC(P,\pi)$, where $\pi$ is $[4]=T_0\cup T_1$.

The main result of this paper determines the order of magnitude of the size of the largest $\cC(P,\pi)$-clique in $\rm^{(r)}_{n}$ for any cube $\cC(P,\pi)$.
\begin{theorem}\label{fullhouse} Let $r\ge2$, let $P$ be a collectable $r$-pattern with a composition $\lambda$ having $s$ components, and let $\pi$ be a partition of $[s]$ into $t+1$ nonempty parts. Then, a.a.s., $$z_{\cC(P,\pi)}(\rm^{(r)}_{n})=\Theta(n^{1/(r-t)}).$$	
\end{theorem}
\noindent Again, note that the threshold depends only on $r$ and $t$ and not on $\lambda$ or $\pi$. In particular,  we get $z_{\cP(\lambda)}(\rm^{(r)}_{n})=\Theta(n^{1/(r-s+1)})$ for \emph{any} composition $\lambda=(\lambda_1,\dots, \lambda_s)$ of $r$ (as then $\pi$ consists of $s=t+1$ singletons).  Notice also that for $s=r$, that is, when $P$ is $r$-partite, we recover a weaker version of Theorem~\ref{thm:random_r-partite}.
 Finally, the case of harmonious pairs in Theorem~\ref{thm:random_two_patterns_gen} is a special instance of Theorem~\ref{fullhouse}.

\medskip

{\bf Example~\ref{12}, continued.}
 Let us look again at Example~\ref{12}. We have there $r=12$, $s=8$, and $t=2$. So,  $z_{\cC(P,\pi)}(\rm^{(12)}_{n})=\Theta(n^{1/10})$, while $z_{\cP(\lambda)}(\rm^{(12)}_{n})=\Theta(n^{1/5})$, not only for $\lambda=(1,2,1,3,2,1,1,1)$, but for any composition with exactly $8$ components summing up to $12$.

\subsection{Variations on the cube}

Interestingly, as we will see, removing any one pattern from a cube causes the size of the largest clique to decrease substantially. Below we state an even more general result. Given a cube $\cC(P,\pi)$, let $\pi'$ be a refinement of $\pi$ in which only the set $T_0$ becomes sub-partitioned, that is, for some $1\le u\le |T_0|-1$,  $\pi'$ is of the form
$$[s]=(T_0'\cup\cdots\cup T_u')\cup T_1\cup\cdots \cup T_t,$$
 where again $1\in T_0'$ and $T_1',\dots, T_u'$ are non-empty. Then, clearly, $\cC(P,\pi)\subseteq \cC(P,\pi')$.

\begin{prop}\label{fullhouse-1} Under the assumptions of Theorem~\ref{fullhouse}, let, in addition, $1\le u\le s-t-1$ and $\pi'$ be as above. Then, a.a.s.,
$$z_{\cC(P,\pi')\setminus  \cC(P,\pi)}(\rm^{(r)}_{n})=O\left(n^{\frac 1{r-u-t+1/u}}\right).$$
\end{prop}
\noindent Note that for $t=0$ we get $z_{\cC(P,\pi')\setminus\{P\}}(\rm^{(r)}_{n})=O\left(n^{\frac 1{r-u+1/u}}\right)=o\left(n^{\frac 1{r-u}}\right)$, which means that the quantity for a cube with one pattern removed is of a smaller order of magnitude than that for the cube itself. For instance, for the cube in Example~\ref{12} with pattern $P$ removed this bound becomes $O(n^{2/21})=o(n^{1/10})$.

 When $u=1$, we get $\cC(P,\pi')\setminus  \cC(P,\pi)=\cC(Q,\pi)$, where $Q$ is $P$ with $P_1'$ flipped,
 and $P_1'$ is the second mega-block related to $\pi'$. For instance, if in Example~\ref{12} we change the partition $\pi$ to $\pi'$ given by $[8]=\{1,3\}\cup\{5\}\cup\{2,8\}\cup \{4,6,7\}$, then $\cC(P,\pi')\setminus  \cC(P,\pi)=\cC(Q,\pi)$, where
 $Q=AB|BBAA|BA|AAABBB|\textcolor{red}{AABB}|BA|AB|AB$.
  So, it is no surprise that the exponent $\tfrac 1{r-u-t+1/u}$ becomes $\tfrac1{r-t}$, the same as in Theorem~\ref{fullhouse}.
But for $u\ge2$ we do go beyond the concept of cubes.

 \begin{exmp}\label{12+12} Consider an extension of Example~\ref{12}.
 Let $P$ be the 12-pattern defined in \eqref{P} and consider the partition  $\pi$: $[8]=\{1,3,5\}\cup\{2,8\}\cup \{4,6,7\}$. Now, let us partition $\{1,3,5\}$ into singletons obtaining partition $\pi'$ given by $[8]=\{1\}\cup\{3\}\cup\{5\}\cup\{2,8\}\cup \{4,6,7\}$. Note that $\cC(P,\pi')$
 is a 4-cube of size $2^4=16$ but
 $\cC(P,\pi')\setminus  \cC(P,\pi)$ is not a cube. It consists of $3\cdot 4=12$ patterns obtained from those in \eqref{1}-\eqref{4} by flipping segments $P_1'$ and $P_2'$ (either one of them, or both).

 The upper bound on $z_{\cC(P,\pi')\setminus  \cC(P,\pi)}(\rm^{(r)}_{n})$, by Proposition~\ref{fullhouse-1} with $r=12$ and $t=u=2$, is $O(n^{2/17})$,
 and is smaller than $n^{1/8}$, the order of $z_{\cC(P,\pi')}(\rm^{(r)}_{n})$ deduced from Theorem~\ref{fullhouse} with $r=12$ and $t=4$. At the same time, assuming that the bound determines the actual order of magnitude of $z_{\cC(P,\pi')\setminus  \cC(P,\pi)}(\rm^{(r)}_{n})$, it is much bigger than $n^{1/10}$ -- the order of $z_{\cC(P,\pi)}(\rm^{(r)}_{n})$.
 \end{exmp}

 Another special case of Proposition~\ref{fullhouse-1} is when all sets $T_1,\dots,T_t$ and $T_0',\dots,T_u'$ are singletons. Then  $u=|T_0|-1=s-t-1$ and $\cC(P,\pi')=\cP(\lambda_P)$, the set of all $r$-patterns harmonious to $P$. In this case, $z_{\cP(\lambda_P)\setminus  \cC(P,\pi)}(\rm^{(r)}_{n})=O\left(n^{\frac 1{r-s+1+1/u}}\right)$, which for  $r$-partite patterns, that is, when $\lambda_P=(1,\dots,1)$,  $s=r$ and $u=r-t-1$, becomes
 $z_{\cR^{(r)}\setminus  \cC(P,\pi)}(\rm^{(r)}_{n})=O\left(n^{1-\frac{1}{r-t}}\right)$.

 \begin{exmp}\label{28} Let $P=(AB)^6$ and $T_0=\{1,2,3,4\}$, $T_1=\{5\}$, and $T_2=\{6\}$. Then $\cC(P,\pi)=\{(AB)^4ABAB,(AB)^4ABBA,(AB)^4BAAB,(AB)^4BABA\}$, while with $T_0'=\{1\}$, $T_1'=\{2\}$, $T_2'=\{3\}$, $T_3'=\{4\}$,
 we have $\cC(P,\pi')=\cR^{(6)}$. (So, the set $\cR^{(6)}\setminus\cC(P,\pi)$ has size $2^5-2^2=28$.) By Proposition~\ref{fullhouse-1} with $r=6$, $t=2$, and $u=3$, we infer that $z_{\cR^{(6)}\setminus\cC(P,\pi)}(\rm^{(r)}_{n})=O(n^{3/4})$.
 \end{exmp}

\medskip

At the moment, only for this last case we are able to match the upper bound from Proposition~\ref{fullhouse-1} with a corresponding lower bound.

\begin{prop}\label{thm:random_three_patterns_stronger}
	Let $P\in\cR^{(r)}$ be an $r$-partite $r$-pattern. Let $\pi$ be a partition $[r]=T_0\cup\cdots\cup T_t$ such that $1\in T_0$ and all other parts $T_i$ are singletons. Then, a.a.s.,
\[
z_{\cR^{(r)}\setminus \cC(P,\pi)}(\rm^{(r)}_{n})=\Theta\left(n^{1-\frac{1}{r-t}}\right).
\]
\end{prop}

Taking $t=0$, that is, $T_0=[r]$, as mentioned earlier, we exclude just the sole pattern $P$ from~$\cR^{(r)}$.  With a little help from our friends - the authors of \cite{AJKS} - we are able to  strengthen Proposition~\ref{thm:random_three_patterns_stronger} in this special case as follows.

\begin{cor}\label{friends}
For $r\ge3$, let $P\in\cR^{(r)}$ be an $r$-partite $r$-pattern and let $\cP$ be a family of $r$-patterns such that
$$\cR^{(r)}\setminus\{P\}\subseteq \cP\subseteq \cP^{(r)}\setminus\{P\}.$$
Then, a.a.s., $z_{\cP}(\rm^{(r)}_{n})=\Theta\left(n^{1-\frac{1}{r}}\right).$
\end{cor}

In particular, for $r=3$ and  any set $\cP$  which includes precisely three patterns from $\cR^{(3)}=\{P_6,P_7,P_8,P_9\}$,  a.a.s., $z_{\cP}(\rm^{(3)}_{n})=\Theta\left(n^{2/3}\right).$
We believe that when $|\cP\cap \cR^{(3)}|=q$, $q=1,2$, then, a.a.s., $z_{\cP}(\rm^{(3)}_{n})=\Theta\left(n^{1/(3-q+1)}\right).$ One instance of the case $q=1$ is presented in the next subsection.

\subsection{Dyck patterns}\label{sec_dyck}


A \emph{Dyck word} is a binary word in which every prefix has at least as many letters $A$ as $B$. A pattern $P$ is called a \emph{Dyck pattern} if its representing word is a Dyck word.
For $r\ge2$, let $\cP^{(r)}_{Dyck}$ be the family of all Dyck $r$-patterns.
It is well known (cf. \cite{StanleyCatalan}) that $|\cP^{(r)}_{Dyck}|=C_r$, where $C_r=\tfrac1{r+1}\binom{2r}r$ is the Catalan number.
For instance,
$$ \mathcal{P}^{(3)}_{Dyck}=\{AABABB,AAABBB,AABBAB,ABAABB,ABABAB\}=\{P_0,P_1,P_3,P_5,P_9\}.$$
Note that every two patterns in $\mathcal{P}^{(3)}_{Dyck}$ are mismatched.


\begin{prop}\label{1-over-r}
For $r\ge2$, let $\cP:=\cP^{(r)}_{Dyck}$ be the family of all Dyck patterns in $\cP^{(r)}$. Then, a.a.s., $z_{\mathcal P}(\rm^{(r)}_{n})=\Theta(n^{1/r}).$
\end{prop}

It means that the order of magnitude of $z_{\mathcal P}(\rm^{(r)}_{n})$ for the entire family of Dyck patterns in $\cP^{(r)}$ is as small as for any single collectable pattern (cf. Theorem~\ref{thm:random_one_pattern}).

\subsection{Organization}

Most proofs, along with essential tools and preliminary observations, are presented in Sections~\ref{section:random-main} and~\ref{comtool}. In particular, Section~\ref{section:random} contains the proof of Theorem~\ref{thm:random_r-partite} and Section~\ref{generic} -- the proofs of upper bounds in Corollary~\ref{friends}, Proposition~\ref{1-over-r}, and  Theorem~\ref{fullhouse}. (Note that, by monotonicity, the lower bounds in Corollary~\ref{friends} and Proposition~\ref{1-over-r} follow, respectively, from Proposition~\ref{thm:random_three_patterns_stronger} and Theorem~\ref{thm:random_one_pattern}.)

In Sections~\ref{pos},~\ref{pap}, and~\ref{trace}, we give the proofs of, respectively, Proposition~\ref{thm:random_three_patterns_stronger} (lower bound), Proposition~\ref{fullhouse-1} (which yields the upper bound in Proposition~\ref{thm:random_three_patterns_stronger}), and the proof of  Theorem~\ref{thm:random_two_patterns_gen} (mismatch case).

Sections~\ref{non-collect} and~\ref{fhlb}, in turn, are devoted to the remaining proofs, that is, proofs of, respectively, Theorem~\ref{thm:random_two_patterns_gen}, the non-collectable case, and Theorem~\ref{fullhouse}, the lower bound. (Recall that the harmonious case of Theorem~\ref{thm:random_two_patterns_gen} follows from Theorem~\ref{fullhouse}.)
We conclude the paper with some open problems and directions of further research.

\medskip

This paper is a full (journal) version of a conference extended abstract \cite{DGPR-LATIN}. Compared with \cite{DGPR-LATIN}, it contains new results, most importantly Theorem~\ref{fullhouse}, but also Propositions~\ref{fullhouse-1},~\ref{thm:random_three_patterns_stronger}, and~\ref{1-over-r}, as well as some proofs which were omitted in \cite{DGPR-LATIN}, and several new figures and examples.

\section{Concatenations and products of matchings}\label{defop}

In this short section we define  operations on matchings which will be used throughout the paper.

 The first operation  ``glues'' together disjoint matchings of equal size, but possibly of distinct dimensions. For $k,r_0,r_1\ge1$, let $V$ be an ordered set of size $|V|=(r_0+r_1)k$. Further, let $M_0=\{e_1,\dots,e_k\}$ be an $r_0$-matching on vertex set $U_0$ and  $M_1=\{f_1,\dots,f_k\}$ be an $r_1$-matching on  vertex set $U_1$, with $U_0\cup U_1\subseteq V$ and $U_0\cap U_1=\emptyset$ (the edges of the matchings are numbered in the order of their left-ends). For a permutation $\sigma:[k]\to[k]$  define the \emph{$\sigma$-concatenation} $M_0\overset{\sigma}{-} M_1$ of $M_0$ and $M_1$ as the $(r_0+r_1)$-matching with the edge set $\{e_1\cup f_{\sigma(1)},\dots, e_k\cup f_{\sigma(k)}\}$. This definition can be extended naturally to $(\sigma_1,\dots,\sigma_h)$-concatenations of $h+1$ matchings
$M_0\overset{\sigma_1}{-} M_1\overset{\sigma_2}{-}\cdots \overset{\sigma_h}{-} M_{h}$.
For instance, having in addition an $r_2$-matching $M_2=\{g_1,\dots,g_k\}$ on vertex set $U_2\subset V$ disjoint from $U_0\cup U_1$,
 $M_0\overset{\sigma_1}{-} M_1\overset{\sigma_2}{-} M_2$
 has the edge set $\{e_1\cup f_{\sigma_1(1)}\cup g_{\sigma_2(1)},\dots, e_k\cup f_{\sigma_1(k)}\cup g_{\sigma_2(k)}\}$.

 Next, define the \emph{product} of $M_0,\dots,M_{h}$ as

 \begin{equation}\label{times}
 M_0\times M_1\times \cdots \times M_{h}=\bigcup_{(\sigma_1,\ldots,\sigma_h)}M_0\overset{\sigma_1}{-} M_1\overset{\sigma_2}{-}\cdots \overset{\sigma_h}{-}M_{h}
 \end{equation}
 where the set union runs over all $h$-tuples of permutations of $[k]$. Equivalently,
$$M_0\times M_1\times \cdots \times M_{h}= \{e^{(0)}\cup e^{(1)}\cup\ldots\cup e^{(h)}:e^{(0)}\in M_0,e^{(1)}\in M_1,\ldots,e^{(h)}\in M_{h}\}.$$
\noindent Note that $M_0\times M_1\times \cdots \times M_{h}$ is a union of $(k!)^h$ matchings, each of size $k$, and
 \begin{equation}\label{edges}
|M_0\times M_1\times \cdots \times M_{h}|=k^h.
\end{equation}

Let us illustrate these operations with a couple of examples.

\begin{exmp}\label{concat}
Let $V=[15]=U_0\cup U_1\cup U_2$, where $U_0=\{1,2,3\},U_1=\{4,\dots,9\}$, and $U_2=\{10,\dots,15\}$, and let $M_0=ABC$, $ M_1=AABBCC$ and $M_2=ABCCBA$ be matchings on sets $U_0,U_1$, and $U_2$, respectively. Further, let $\sigma_1=231$ and $\sigma_2=132$ be permutations of $[3]$. In letter notation, $\sigma_1$ assigns $A\to B$, $B\to C$, and $C\to A$, while $\sigma_2$ assigns $A\to A$, $B\to C$ and $C\to B$. Thus, using fresh letters $D,E,F$,
$$M:=M_0\overset{\sigma_1}{-} M_1\overset{\sigma_2}{-} M_2=DEF|FFDDEE|DFEEFD.$$
 (Here the $D$-edge combines the $A$-edge from $M_0$, the $B$-edge from $M_1$ and the $A$-edge from $M_2$, and the $E$-edge and $F$-edge are defined similarly, based on the two permutations.)

 Note that $M$ is a $(\cC,\pi)$-clique of size three, where $P=AB|AABB|AB|BA$ and $\pi=\{\{1\},\{2\},\{3,4\}\}$. Indeed, letters $D$ and $E$ form pattern $P$, letters $D$ and $F$ form pattern $P_{\{1\}}$, and $E$ and $F$ form pattern $P_{\{1,2\}}$ (in the notation of Section~\ref{bool}).

 Taking another pair of permutations, say $\sigma'_1=213$ and $\sigma'_2=123$, we obtain another $(\cC,\pi)$-clique of size three,
 $$M':=M_0\overset{\sigma_1}{-} M_1\overset{\sigma_2}{-} M_2=DEF|EEDDFF|DEFFED,$$
 which shares the edge $D$, or $\{1,6,7,10,15\}$, with $M$.
\end{exmp}

\begin{exmp}\label{product}
Let $V=[12]=U_0\cup U_1\cup U_2$, where $U_0=\{1,2,3,4\},U_1=\{5,6,7,8\}$, and $U_2=\{9,10,11,12\}$, and let $M_0=AABB$, $M_1=ABBA$, and $M_2=ABAB$ be matchings on sets $U_0,U_1$, and $U_2$, respectively.
To get the product $M_0\times M_1\times M_{2}$ one needs to consider all four pairs of permutations $\sigma_1,\sigma_2\in\{12,21\}$. Computing the concatenation $M_0\overset{\sigma_1}{-} M_1\overset{\sigma_2}{-} M_2$ for each pair, we obtain four 6-matchings on $[12]$:
$CCDD|CDDC|CDCD$, $CCDD|CDDC|DCDC$, $CCDD|DCCD|CDCD$, and $CCDD|DCCD|DCDC$ which together form the 2-cube $\cC(P,\pi)$, where $P=M_0M_1M_2$ and $\pi=\{\{1,2\},\{3,4\},\{5,6\}\}$.

To see the resulting 6-uniform hypergraph, it is better to return to vertices instead of letters. For instance, $CCDD|CDDC|CDCD$ encodes the pair of edges $\{1,2,5,8,9,11\}$ and $\{3,4,6,7,10,12\}$. Thus,
\begin{align*}
M_0\times M_1\times M_{2}=\{& \{1,2,5,8,9,11\}, \{3,4,6,7,10,12\}, \{1,2,5,8,10,12\}, \{3,4,6,7,9,11\},\\&
 \{1,2,6,7,9,11\},\{3,4,5,8,10,12\},\{1,2,6,7,10,12\},\{3,4,5,8,9,11\} \},
\end{align*}
\noindent see Figure 3.1. This time, the four matchings forming $M_0\times M_1\times M_{2}$ are pairwise edge-disjoint, so the hypergraph is 4-uniform. It should be clear by now that after flipping the $A$'s and  the $B$'s around in either $M_1$ or $M_2$ or both, we would end up with the very same hypergraph as the product.
\end{exmp}

\begin{figure}[ht]
	\captionsetup[subfigure]{labelformat=empty}
	\begin{center}		
	
	\scalebox{1}
{
\centering
\begin{tikzpicture}
[line width = .5pt,
vtx/.style={circle,draw,red,very thick,fill=red, line width = 1pt, inner sep=0pt},
]

    \node[vtx] (1) at (0,0) {};
    \node[vtx] (13) at (11.8,0) {};

    \draw[line width=0.3mm, color=black]  (1) -- (13);

    \node[vtx] (1a) at (0+0.5,0) {};
    \node[vtx] (1b) at (1+0.5,0) {};
    \node[vtx] (1c) at (2+0.5,0) {};

    \node[vtx] (2a) at (3+0.5,0) {};
    \node[vtx] (2b) at (4+0.5,0) {};
    \node[vtx] (2c) at (5+0.5,0) {};

    \node[vtx] (3a) at (6+0.5,0) {};
    \node[vtx] (3b) at (7+0.5,0) {};
    \node[vtx] (3c) at (8+0.5,0) {};

    \node[vtx] (4a) at (9+0.5,0) {};
    \node[vtx] (4b) at (10+0.5,0) {};
    \node[vtx] (4c) at (11+0.5,0) {};

    \draw[line width=0.5mm, color=red, outer sep=2mm] (1b) arc (0:180:0.5);
    \draw[line width=0.5mm, color=red, outer sep=2mm] (2a) arc (0:180:0.5);
    \draw[line width=0.5mm, color=red, outer sep=2mm] (3a) arc (0:180:0.5);
    \draw[line width=0.5mm, color=red, outer sep=2mm] (3b) arc (0:180:1.5);
    \draw[line width=0.5mm, color=red, outer sep=2mm] (4c) arc (0:180:1);
    \draw[line width=0.5mm, color=red, outer sep=2mm] (4b) arc (0:180:1);

    \draw[line width=0.5mm, color=blue, outer sep=2mm] (2b) arc (0:180:0.5);
    \draw[line width=0.5mm, color=blue, outer sep=2mm] (2b) arc (0:180:1.5);
    \draw[line width=0.5mm, color=blue, outer sep=2mm] (2c) arc (0:180:1.0);
    \draw[line width=0.5mm, color=blue, outer sep=2mm] (2c) arc (0:180:2);
    \draw[line width=0.5mm, color=blue, outer sep=2mm] (3c) arc (0:180:0.5);
    \draw[line width=0.5mm, color=blue, outer sep=2mm] (3c) arc (0:180:1);
    \draw[line width=0.5mm, color=blue, outer sep=2mm] (4a) arc (0:180:1);
    \draw[line width=0.5mm, color=blue, outer sep=2mm] (4a) arc (0:180:1.5);

    \fill[fill=black, outer sep=1mm]   (1a) circle (0.1);
    \fill[fill=black, outer sep=1mm]   (1b) circle (0.1);
    \fill[fill=black, outer sep=1mm]   (1c) circle (0.1);

    \fill[fill=black, outer sep=1mm]   (2a) circle (0.1);
    \fill[fill=black, outer sep=1mm]   (2b) circle (0.1);
    \fill[fill=black, outer sep=1mm]   (2c) circle (0.1);

    \fill[fill=black, outer sep=1mm]   (3a) circle (0.1);
    \fill[fill=black, outer sep=1mm]   (3b) circle (0.1);
    \fill[fill=black, outer sep=1mm]   (3c) circle (0.1);

    \fill[fill=black, outer sep=1mm]   (4a) circle (0.1);
    \fill[fill=black, outer sep=1mm]   (4b) circle (0.1);
    \fill[fill=black, outer sep=1mm]   (4c) circle (0.1);

    \node[color=black] at (0+0.5,-0.5) {$1$};
    \node[color=black] at (1+0.5,-0.5) {$2$};
    \node[color=black] at (2+0.5,-0.5) {$3$};
    \node[color=black] at (3+0.5,-0.5) {$4$};
    \node[color=black] at (4+0.5,-0.5) {$5$};
    \node[color=black] at (5+0.5,-0.5) {$6$};
    \node[color=black] at (6+0.5,-0.5) {$7$};
    \node[color=black] at (7+0.5,-0.5) {$8$};
    \node[color=black] at (8+0.5,-0.5) {$9$};
    \node[color=black] at (9+0.5,-0.5) {$10$};
    \node[color=black] at (10+0.5,-0.5) {$11$};
    \node[color=black] at (11+0.5,-0.5) {$12$};

\end{tikzpicture}
}
	
	\end{center}
	
	\caption{The product $M_0\times M_1\times M_{2}$ from Example~\ref{product}. Each of its eight $6$-vertex edges is depicted by a red-blue-red-blue-red path. There are four edges containing $\{1,2\}$ and four edges containing $\{3,4\}$.}
	\label{MMM}
	
\end{figure}

\section{Probabilistic tools} In this  section we present a couple of probabilistic tools and apply them right away to prove some of our results.\label{section:random-main}

\subsection{ Concentration inequalities for random matchings}\label{section:random}

Recall that $\rm^{(r)}_{n}$ denotes a random ordered $r$-uniform matching of size $n$, that is, a~matching picked uniformly at random out of the set $\cM_n^{(r)}$ of all
\[
\alpha^{(r)}_n:=\frac{(rn)!}{(r!)^n\, n!}
\]
ordered $r$-matchings on the set $[rn]$. For future reference note that the probability that a given $r$-element subset $e$ of $[rn]$ is an edge of $\rm^{(r)}_{n}$ is
\begin{equation}\label{1edge}
\PP(e\in \rm^{(r)}_{n})=\frac{\alpha^{(r)}_{n-1}}{\alpha^{(r)}_n}=\frac1{\binom{rn-1}{r-1}}.
\end{equation}

There is another, equivalent way of drawing $\rm^{(r)}_{n}$, the \emph{permutational scheme}.
Indeed, let $\pi$ be a permutation of $[rn]$. It can be chopped off into an $r$-matching \[M_\pi:=\left\{\{\pi(1),\dots,\pi(r)\}, \{\pi(r+1),\dots,\pi(2r)\},\dots,\{\pi(rn-r+1),\dots,\pi(rn)\}\right\}\] and, clearly, there are exactly $(r!)^n n!$ permutations $\pi$ yielding the same matching.
This amounts to a basic fact that the probability of choosing a matching with given features while picking it uniformly at random is the same in the case when we pick a random permutation instead, and then chop it into the matchings' edges.

We will benefit from this straightforward observation, as it allows direct applications of
high concentration results for random permutations in the context of ordered matchings. By high concentration we mean that a random variable is very close to its expectation with probability very close to 1.

 One of them is the Azuma-Hoeffding inequality for random permutations (see, e.g., Lemma 11 in~\cite{FP} or  Section 3.2 in~\cite{McDiarmid98}).
 By swapping two elements in a permutation $\pi_1$ we mean fixing two indices $i<j$ and creating a new permutation $\pi_2$ with $\pi_2(i)=\pi_1(j)$, $\pi_2(j)=\pi_1(i)$, and $\pi_2(\ell)=\pi_1(\ell)$ for all $\ell\neq i,j$. Let $\Pi_{N}$ denote a permutation selected uniformly at random from all $N!$ permutations of $[N]$.

\begin{lemma}[\cite{FP,McDiarmid98}]\label{azuma}
 Let $h(\pi)$ be a function defined on the set of all permutations of order $N$ which, for some positive constants $c$, satisfies the following Lipschitz-type condition:
 \begin{enumerate}
 \item[(L)] if $\pi_2$ is obtained from $\pi_1$ by swapping two elements, then $|h(\pi_1)-h(\pi_2)|\le c$.
 \end{enumerate}
Then, for every $\eta>0$,
\[
\PP(|h(\Pi_{N})-\E[h(\Pi_{N})]|\ge \eta)\le 2\exp\{-2\eta^2/(c^2N)\}.
\]
\end{lemma}
Lemma~\ref{azuma} yields sharp concentration only when $\E[h(\Pi_{N})]\gg \sqrt N$.
This is not sufficient for the proof of the lower bound in Theorem~\ref{fullhouse} and therefore we will use there
 a Talagrand's concentration inequality from \cite{Talagrand}, which can be applied whenever the expectation tends to infinity. The drawbacks are that it requires another, so called  ``certificate'' condition and that it is expressed in terms of the median rather than the expectation.

 We quote here a slightly simplified version from \cite{LuczakMcDiarmid} (see also \cite{McDiarmid}).

\begin{lemma}[\cite{LuczakMcDiarmid,McDiarmid}]\label{tala}
	Let $h(\pi)$ be a function defined on the set of all permutations of order $N$ which, for some positive constants $c$ and $d$, satisfies
	\begin{enumerate}
		\item[(L)] if $\pi_2$ is obtained from $\pi_1$ by swapping two elements, then $|h(\pi_1)-h(\pi_2)|\le c$;
		\item[(C)] for each $\pi$ and $s>0$, if $h(\pi)=s$, then in order to show that $h(\pi)\ge s$, one needs to specify only at most $ds$ values $\pi(i)$.
	\end{enumerate}
	Then, for every $\eps>0$,
	$$\PP(|h({\Pi}_N)-m|\ge \eps m)\le4 \exp(-\eps^2 m/(32dc^2)),$$
	where $m$ is the median of $h({\Pi}_N)$.
\end{lemma}

We will now employ Lemma~\ref{azuma} to prove Theorem~\ref{thm:random_r-partite}. The second concentration lemma, Lemma~\ref{tala} will be  used in the main proof in Section~\ref{fhlb}.

\begin{proof}[Proof of Theorem~\ref{thm:random_r-partite}]

For given disjoint subsets $A_1,\dots,A_r \subseteq [rn]$ we say that an $r$-edge $e$ \emph{spans} these sets if $|e\cap A_i|=1$ for each $1\le i\le r$ (see Figure\ \ref{span}).

\begin{figure}[ht]
\captionsetup[subfigure]{labelformat=empty}
\begin{center}

\scalebox{1}
{
\centering
\begin{tikzpicture}
[line width = .5pt,
vtx/.style={circle,draw,black,very thick,fill=black, line width = 1pt, inner sep=0pt},
]

    \node[vtx] (1) at (0,0) {};
    \node[vtx] (2) at (2,0) {};
    \node[vtx] (3) at (4,0) {};
    \coordinate (34) at (6,1) {};
    \node[vtx] (4) at (8,0) {};
    \node[vtx] (5) at (10,0) {};

    \draw[line width=0.3mm, color=lightgray]  (1) -- (5);

    \draw[line width=2.5mm, color=gray, line cap=round]  (1.5,0) -- (2.5,0);
    \draw[line width=2.5mm, color=gray, line cap=round]  (3.75,0) -- (4.75,0);
    \draw[line width=2.5mm, color=gray, line cap=round]  (7.25,0) -- (8.25,0);

    \node[color=gray] at (2.6,-0.35) {$A_1$};
    \node[color=gray] at (4.85,-0.35) {$A_2$};
    \node[color=gray] at (8.35,-0.35) {$A_3$};
    \node[color=blue] at (4,0.7) {$e$};

    \draw[line width=0.5mm, color=blue, outer sep=2mm] plot [smooth, tension=2] coordinates {(3) (34) (4)};
    \draw[line width=0.5mm, color=blue, outer sep=2mm] (3) arc (0:180:1);

    \fill[fill=black, outer sep=1mm]   (1) circle (0.1) node [below] {$1$};
    \fill[fill=black, outer sep=1mm]   (2) circle (0.1);
    \fill[fill=black, outer sep=1mm]  (3) circle (0.1);
    \fill[fill=black, outer sep=1mm]   (4) circle (0.1);
    \fill[fill=black, outer sep=1mm]   (5) circle (0.1) node [below] {$3n$};

\end{tikzpicture}
}

\end{center}

\caption{A 3-edge spanning three sets.}
\label{span}
\end{figure}

 Recall that $\cR^{(r)}$ stands for the set of all $r$-partite $r$-patterns and let $M$ be an $\cR^{(r)}$-clique  on the vertex set $V(M)\subset [rn]$, $|V(M)|=r|M|$. Since $M$ is  $r$-partite itself, the vertices of $M$ can be partitioned into $r$ sets $A_i \subseteq [rn]$ with $|A_i|=|M|$,  $1\le i\le r$, and such that for each $i=1,\dots,r-1$, all elements of $A_i$ appear before elements of $A_{i+1}$, and each edge of $M$ spans the sets $A_1,\dots,A_r$. Clearly, there is a partition $[rn]=B_1\cup\cdots\cup B_r$ into pairwise disjoint blocks such that for each $i=1,\dots,r$ we have $A_i\subseteq B_i$.
 Consequently, in estimating the largest size $z_{\cR^{(r)}}(\rm^{(r)}_{n})$ of an $r$-partite sub-matching of $\rm^{(r)}_{n}$ it suffices to consider only matchings whose edges span all sets $B_1,\dots,B_r$ of such a partition.

For a given partition $\mathbf{B}$ of $[rn]$ into blocks $B_1,\dots,B_r$ define a random variable $X_{\mathbf{B}}$ that counts the number of edges of  $\rm^{(r)}_{n}$ spanning the sets $B_i$. Note that the number of choices for $\mathbf{B}$ is $O(n^{r-1})$ and that $z_{\cR^{(r)}}(\rm^{(r)}_{n})=\max_{\mathbf{B}}X_{\mathbf{B}}$.
 We will show that, a.a.s., for every $\mathbf{B}$ we have $X_{\mathbf{B}} \le \frac{(r-1)!}{r^{r-1}}n(1+o(1))$ and for some $\mathbf{B}$ this bound is asymptotically achieved.

To this end, fix a partition $\mathbf{B}$ with $|B_i|=b_i$, $i=1,\dots,r$, where $\sum_{i=1}^rb_i=rn$. For each $r$-element set $e$ which spans the sets $B_1,\dots,B_r$, let $I_e=1$ if $e\in\rm^{(r)}_{n}$ and $I_e=0$ otherwise.
Then $X_{\mathbf{B}}=\sum_eI_e$ and  by~\eqref{1edge} and the linearity of expectation
\[
\E(X_{\mathbf{B}}) =\sum_e\E(I_e)
=   \frac{\prod_{i=1}^rb_i}{\binom{rn-1}{r-1}}.
\]
Note that for the equipartition $\mathbf{B^{*}}$ with all $b_i=n$, we get
\begin{equation}\label{B*}
\E(X_{\mathbf{B^{*}}})=  \frac{n^r}{\binom{rn-1}{r-1}} \sim \frac{(r-1)!}{r^{r-1}}n.
\end{equation}

It turns out that for all $\mathbf{B}$, the above quantity is an upper bound on $\E(X_{\mathbf{B}})$.
Indeed, by the inequality of arithmetic and geometric means, we get
\[
\prod_{i=1}^r |B_i|
\le \left( \frac{\sum_{i=1}^r |B_i|}{r} \right)^r
= n^r.
\]
Thus, for all $\mathbf{B}$,
\begin{equation}\label{ineq:X_B}
\E(X_{\mathbf{B}}) \le  \frac{n^r}{\binom{rn-1}{r-1}} \le (1+o(1))\frac{(r-1)!}{r^{r-1}}n.
\end{equation}

Now we are going to apply Lemma~\ref{azuma} to $X_{\mathbf{B}}$ for each $\mathbf{B}$ separately.
With respect to the Lipschitz condition the following observation is crucial for us.
Let $\pi_1$ be a permutation of $[rn]$ and let permutation $\pi_2$ be obtained from $\pi_1$ by swapping two elements. This way we can destroy (or create) at most two edges which may jointly contribute at most $2$ to $X_{\mathbf{B}}$. Hence, each random variable $X_{\mathbf{B}}$ satisfies the Lipschitz condition with $c=2$. Thus, Lemma~\ref{azuma} applied with $N=rn$, $c=2$, and $\eta=\sqrt{\omega n\log n}$ (where $\omega:=\omega(n)\to\infty$) yields that
$$\PP\left(|X_{\mathbf{B}}-\E(X_{\mathbf{B}})| \ge \sqrt{\omega n\log n}\right)\le 2n^{-\omega/2r}.$$
 This, together with~\eqref{ineq:X_B} and the union bound over all partitions $\mathbf{B}$, implies that, a.a.s., $z_{\cR^{(r)}}(\rm^{(r)}_{n})=\max_{\mathbf{B}}X_{\mathbf{B}}$ is at most $(1+o(1))\frac{(r-1)!}{r^{r-1}}n$. Moreover, by the same token, in view of~\eqref{B*}, this upper bound is asymptotically achieved by $X_{\mathbf{B^{*}}}$, which completes the proof of~Theorem~\ref{thm:random_r-partite}.
 \end{proof}

\subsection{Generic first moment method}\label{generic}

For most  upper bound proofs in this paper we use a standard first moment method. Here we state a generic, technical ``meta-statement'' which will be applied throughout. For a set of $r$-patterns $\cP$, let $a_{\mathcal P}(k)$ denote the  number of $\mathcal P$-cliques among the matchings in $\cM_k^{(r)}$, that is, the  number of distinct $\mathcal P$-cliques of size $k$ one can build on a given ordered set of $rk$ vertices.

\begin{lemma}\label{meta} For $r\ge2$, let $\cP$ be a set of $r$-patterns.
If there are  constants $C>0$ and  $0\le x<r$ such that  $a_{\mathcal P}(k)\le C^{k}k^{kx}$ for all $k\ge2$, then, a.a.s., $z_{\cP}(\rm^{(r)}_{n})=O(n^{1/(r-x)})$.
\end{lemma}

\proof  For each $k\ge2$, let $X_k$ be the number of $\cP$-cliques of size $k$ in $\rm^{(r)}_n$.
In order to compute the expectation of $X_k$, one has to first choose a set $S$ of $rk$ vertices (out of all $rn$ vertices) on which a $\cP$-clique will be planted. Formally, for each $S\in\binom{[rn]}{rk}$ define an indicator random variable $I_S=1$ if there is a $\cP$-clique on $S$ in $\rm^{(r)}_n$, and $I_S=0$ otherwise.
Thus,
\[
\E I_S = \Pr(I_S=1)= a_{\mathcal P}(k)\cdot \frac{\alpha^{(r)}_{n-k}}{\alpha^{(r)}_n}
\]
and by the linearity of expectation, noticing that $X_k=\sum_S I_S$,
$$\E X_k=\binom{rn}{rk}a_{\cP}(k)\frac{\alpha^{(r)}_{n-k}}{\alpha^{(r)}_n}=\frac{(r!)^k}{(rk)!} \cdot a_{\cP}(k) \cdot \frac{n!}{(n-k)!}\le
\left( \frac{C' n}{k^{r-x}} \right)^k,
$$
where $C'=Cr!e^r r^{-r}$.
Let $k_0=\lceil C''n^{1/(r-x)}\rceil = \Theta(n^{1/(r-x)})$, where $C''>(C')^{1/(r-x)}$.
  Then, by  Markov's inequality,
\[
\Pr( X_{k_0}\ge1)
\le  \left( \frac{C'n}{(k_0)^{r-x}}\right)^{k_0}
= o(1).
\]
Finally note that  $ X_{k_0}\ge1$ is equivalent to $z_{\cP}(\rm^{(r)}_{n})\ge k_0$. \qed

\medskip

As two quick applications of Lemma~\ref{meta}, we deduce the upper bounds in Corollary~\ref{friends} and Proposition~\ref{1-over-r}. They both rely on existing upper bounds for the number $a_{\cP}(k)$ for the respective families $\cP$.

\begin{proof}[Proof of Corollary~\ref{friends}] First, recall that, by monotonicity, the lower bound in Corollary~\ref{friends}  follows from Proposition~\ref{thm:random_three_patterns_stronger}, so we only need to show the upper bound.

Let $P\in\cR^{(r)}$ and $\cP=\cP^{(r)}\setminus\{P\}$. Then, by \cite[Theorem 1.13]{AJKS}, for some $C>0$
$$a_{\cP}(k)\le C^kk^{(r-1-1/(r-1))k}.$$
This, in turn, implies, via Lemma~\ref{meta} with $x=r-1-1/(r-1)$, that, a.a.s., $z_{\cP}(\rm^{(r)}_{n})=O(n^{1-1/r})$.

\end{proof}

For the next proof, we need a generalization of Dyck words.
For positive integers $d$ and $m$, a \emph{$d$-dimensional Dyck word}  of length $dm$ over a linearly ordered alphabet of size $d$ consists  of $m$ occurrences of each letter,  and,  for each pair of letters $X<Y$, every prefix contains at least as many letters $X$ as $Y$.
They are enumerated by $d$-dimensional Catalan numbers $C^{(d)}_m$ given by formula A060854 of the OEIS data base (cf. \cite{MMoo}), which is of no practical importance to us. What \emph{is} important is the identity
 $C_m^{(d)}=C_d^{(m)}$  valid for all $m$ and $d$.

\begin{proof}[Proof of Proposition~\ref{1-over-r}] Again by monotonicity, the lower bound in  Proposition~\ref{1-over-r} follows from Theorem~\ref{thm:random_one_pattern}.
For the upper bound, observe that, by definition, $\cP^{(r)}_{Dyck}$-cliques of size $k$ are in a one-to-one correspondence with $k$-dimensional Dyck words of length $kr$. Hence, by the above identity, $a_{\cP^{(r)}_{Dyck}}(k)=C_r^{(k)}=C_k^{(r)}\le r^{rk}$, where the upper bound follows from the fact that the number of $r$-dimensional Dyck words of length $rk$ cannot exceed the total number of words of that length over an $r$-element  alphabet. Thus, Lemma~\ref{meta} can be applied with $C=r^r$ and $x=0$, yielding that, a.a.s., $z_{\cP^{(r)}_{Dyck}}(\rm^{(r)}_{n})=O(n^{1/r})$.

\end{proof}

\begin{rem}\label{r2} Returning to Theorem~\ref{thm:random_two_patterns_r=2}, recall our notation, $\cP_1=\{AABB,ABAB\}$ and $\cP_2=\{AABB,ABBA\}$, and note that $\cP_1=\cP_{Dyck}^{(2)}$, so the upper bound on $z_{\cP_1}(\rm^{(r)})$ follows by Proposition~\ref{1-over-r}.
As for $\cP_2$, Stanley \cite{StanleyCatalan} constructed a bijection between $\cP_2$-cliques and $\cP_1$-cliques of the same size, so that   $a_{\cP_2}(k)=a_{\cP_1}(k)=C_k^{(2)}<2^{2k}$, and the bound on $z_{\cP_2}(\rm^{(r)})$ follows by Lemma~\ref{meta}  with $C=4$ and $x=0$.
(Both lower bounds are implied by Theorem \ref{thm:random_one_pattern} and monotonicity.)

\end{rem}

Finally in this section, we prove the upper bound in Theorem~\ref{fullhouse}

\begin{proof}[Proof of Theorem~\ref{fullhouse}, upper bound]
Suppose $P=S_1S_2\cdots S_s$ is an arbitrary collectable $r$-pattern, where $S_i$ are the blocks of the splitting. Let $\lambda_P=(\lambda_1,\dots,\lambda_s)$ be the composition determined by $P$, that is, $r=\lambda_1+\cdots +\lambda_s$ and $\lambda_i=|S_i|/2$. Further,
let $\pi = \{T_0,T_1,\ldots, T_t\}$ be a fixed partition of the set $[s] = T_0\cup T_1\cup \ldots\cup T_t$
with  $1\in T_0$. Finally, for $j=0,\dots,t$, set $P_j$ for the concatenation of all segments $S_i$ with $i\in T_j$.

We are going to show that $a_{\cC(P,\pi)}(k)=(k!)^t$ which, by Lemma~\ref{meta} with $C=1$ and $x=t$, will imply that $z_{\cC(P,\pi)}(\rm^{(r)}_{n})= O\left(n^{1/(r-t)}\right)$.
 To this end, let $[rk]=V_1\cup\cdots\cup V_s$ be a partition of $[rk]$ into consecutive blocks, where $|V_i|=k\lambda_i$, for $i=1,\dots,s$. Let $U_j:=\bigcup_{i\in T_j} V_i$, for $j=0,\dots, t$, and set $r_j:=\sum_{i\in T_j}\lambda_i$.

Let  $K_j$ be a (unique) $P_j$-clique on vertex set $U_j$, $j=0,\dots,t$. Note that $K_j$ consists of $k$ edges of size $r_j$. In particular, we may have $r_j=1$, in which case $K_j$ is a 1-matching. Then, by the definition of a $t$-cube (see Section~\ref{bool}), any $t$-tuple of permutations $(\sigma_1,\sigma_2,\dots,\sigma_t)$ of $[k]$ defines a \emph{distinct}  $\cC(P,\pi)$-clique $K_0\overset{\sigma_1}{-} K_1\overset{\sigma_2}{-}\cdots \overset{\sigma_t}{-} K_t$ on $[rk]$ (for the definition of concatenation see Section~\ref{defop}).

We are going to show that
every $\cC(P,\pi)$-clique $K$ on $[rk]$ is of that form, that is, they are in one-to-one correspondence with the terms in the set union defining $K_0\times K_1\times\cdots\times K_t$ (see~\eqref{times}). Since the number of choices of $(\sigma_1,\sigma_2,\dots,\sigma_t)$ is $(k!)^t$, this will imply that  $a_{\cC(P,\pi)}(k)=(k!)^t$ and, thus, conclude the proof.

 Let $K$ be a $\cC(P,\pi)$-clique on $[rk]$. Then, for  every edge $e\in K$ we have $|e\cap V_i|=\lambda_i$, for  $i=1,\dots,s$. Indeed, suppose it is not true and let $i_0$ be the smallest index $i$ with $|e\cap V_i|\neq\lambda_i$ for some $e\in K$, say $|e\cap V_i|<\lambda_i$. Then there must be another edge $f\in K$ with $|f\cap V_i|>\lambda_i$. But then $e$ and $f$ form a pattern which on positions $2\sum_{i<i_0}\lambda_i+1,\dots,2\sum_{i\le i_0}\lambda_i$ has a different number of $A$'s and $B$'s, so it is not a pattern from $\cC(P,\pi)$, a contradiction.

Consequently,
 for each $j=0,1,\dots,t$, every edge $e\in K$ satisfies $|e\cap U_j|=r_j$.
For $j=0,\dots,t$, let $K[U_j]$ be the ordered $r_j$-matching on $U_j$ consisting of all edges $e\cap U_j$ with $e\in K$. Then,  each $K[U_j]$ makes up the unique $P_j$-clique planted on $U_j$, that is, $K_j$. This is because every pair of edges of $K[U_j]$ forms the same pattern $P_j$ (note that a flip of entire pattern does not change it).
Therefore, indeed,
\begin{equation}\label{KWi}
K=K_0\overset{\sigma_1}{-} K_1\overset{\sigma_2}{-}\cdots \overset{\sigma_t}{-} K_t
\end{equation}
for some  $t$-tuple of permutations $(\sigma_1,\sigma_2,\dots,\sigma_t)$ of $[k]$.

\end{proof}

\begin{exmp}\label{fullfull}
Let us illustrate the above proof with the 2-cube $\cC(P,\pi)$ defined by 5-pattern $P=AB|AABB|AB|BA=S_1S_2S_3S_4$ with $\lambda_P=(1,2,1,1)$ and $s=4$, and by partition $\pi$:  $[4]=\{1\}\cup\{2\}\cup\{3,4\}$. Besides $P_{\emptyset}=P$, this cube consists of $P_{\{1\}}=AB|BBAA|ABBA$, $P_{\{2\}}=AB|AABB|BAAB$ and $P_{\{1,2\}}=AB|BBAA|BAAB$.

Following notation from the above proof,
$P_0=S_1=AB$, $P_1=S_2=AABB$, and $P_2=S_3S_4=ABBA$. For $k=3$, consider the partition $[15]=V_1\cup V_2\cup V_3\cup V_4$, where the consecutive  blocks $V_i$ have sizes $3,\;6,\;3,\;3$, and let $U_0=V_1$, $U_1=V_2$,  $U_2=V_3\cup V_4$.
 Let
 $$K_0=ABC,\quad K_1=AABBCC,\quad\mbox{and}\quad K_2=ABCCBA$$
 be the $P_0$-clique, $P_1$-clique, and $P_2$-clique on the vertex sets, resp., $U_0,U_1$, and $U_2$.
Further, let $\sigma_1=231$ and $\sigma_2=132$ be a pair of permutations of $[3]$.
By Example~\ref{concat}, we  have
$$K_0\overset{\sigma_1}{-} K_1\overset{\sigma_2}{-} K_2=DEF|FFDDEE|DFEEFD$$
which, indeed, is a $\cC(P,\pi)$-clique, as the edges designated by $D$ and $E$ form pattern $P$, by $D$ and $F$ - pattern $P_{\{1\}}$, while $E$ and $F$ - pattern $P_{\{1,2\}}$.

 In the other direction, consider, for instance, the  $\cC(P,\pi)$-clique
 $$K=ABC|BBCCAA|CBAABC.$$ Then, taking $\sigma_1=312$ and $\sigma_2=321$, we get $K=K_0\overset{\sigma_1}{-} K_1\overset{\sigma_2}{-} K_2.$

\end{exmp}

\section{Combinatorial tools}\label{comtool}
In this  section we present several combinatorial tools and apply them right away to prove some of our results.

\subsection{Posets on matchings and proof of Proposition~\ref{thm:random_three_patterns_stronger} (lower bound)}\label{pos} As remarked in \cite{GP2019},  one can often define a partially ordered set (poset) on the edges of a random matching $\rm_n^{(r)}$ and then utilize Dilworth's Lemma, which we state here in a version most suited for us.

 \begin{lemma}[Dilworth, 1950]\label{DL}
 If the largest chain in a finite poset $(X,\prec)$ has size at most $a$, then the largest anti-chain in $(X,\prec)$ has size at least $|X|/a$.
 \end{lemma}

A general template of how to proceed in such scenarios is described in the next lemma.

\begin{lemma}\label{template} Let $\cP\subseteq\cP^{(r)}$ be a set of $r$-patterns and let $N$ be a $\cP$-clique in an $r$-matching~$M$.
Suppose  that for a subset $\cQ\subset \cP$,  a poset $(N,\prec)$  can be defined in such a way that for all $e,f\in N$ we have $e\prec f$ if $\min e<\min f$ and  $e,f$ form an $r$-pattern from $\cQ$.
Then,  $z_{\cQ^c}(M)\ge |N|/z_{\cQ}(M)$, where $\cQ^c=\cP\setminus\cQ$.
\end{lemma}

\proof Note that for such a poset, every chain is a $\cQ$-clique, while every anti-chain is a $\cQ^c$-clique. As $N\subseteq M$, by monotonicity, $z_{\cQ}(N)\le z_{\cQ}(M)$ and $z_{\cQ^c}(N)\le z_{\cQ^c}(M)$.  By Lemma~\ref{DL} applied to $(N,\prec)$ we thus conclude that
$$z_{\cQ^c}(M)\ge z_{\cQ^c}(N)\ge |N|/z_{\cQ}(N)\ge |N|/z_{\cQ}(M).$$  \qed

 \medskip
  In this paper we benefit from Lemma~\ref{template} only in the case of $r$-partite patterns, relying, in addition, on Corollary~\ref{Mpart}. We now deduce the lower bound in Proposition~\ref{thm:random_three_patterns_stronger} from the upper bound in Theorem~\ref{fullhouse}.

  \begin{proof}[Proof of Proposition~\ref{thm:random_three_patterns_stronger}, lower bound]
  Let $P=S_1\cdots S_r$, where $|S_i|=2$ for $i=1,\dots,r$, and let $\pi$ be a partition $[r]=T_0\cup\cdots\cup T_t$ such that $1\in T_0$ and all other parts $T_i$ are singletons. Further, let $T_0^{AB}\subseteq T_0$ consist of all those indices $i$ for which $S_i=AB$, and let $T_0^{BA}=T_0\setminus T_0^{AB}$.  Recall that $R_n^{(r)}$ denote the largest $r$-partite sub-matching of $\rm^{(r)}_{n}$ and that by Corollary~\ref{Mpart} we have $|R_n^{(r)}|=\Theta(n)$, a.a.s. Given two edges $e,f\in R_n^{(r)}$ with $\min e<\min f$, let
 \begin{equation}\label{po3}
e=(x_1<\cdots <x_r)\prec f=(y_1<\cdots <y_r)\quad\mbox{if}\quad\begin{cases} x_i<y_i\quad\mbox{for all}\quad i\in T_0^{AB}
\\x_i>y_i\quad\mbox{for all}\quad i\in T_0^{BA}.\end{cases}
\end{equation}

It is easy to see that in the poset $(R_n^{(r)},\prec)$ we have $e\prec f$ if and only if  $e$ and $f$ form a pattern in $\cC(P,\pi)$, so  a chain is  a $\cC(P,\pi)$-clique and an anti-chain is an $\cR^{(r)}\setminus\cC(P,\pi)$-clique.
   By Theorem~\ref{fullhouse}, a.a.s., $z_{\cC(P,\pi)}(\rm^{(r)}_{n})=O(n^{1/(r-t)})$. Thus, we infer from Lemma~\ref{template} with $N=R_n^{(r)}$ and $M=\rm^{(r)}_{n}$ that, a.a.s., $z_{\cR^{(r)}\setminus\cC(P,\pi)}(\rm^{(r)}_{n})=\Omega(n^{1-1/(r-t)})$.
\end{proof}

   As an illustration of the proof, let us return to Example~\ref{28}, in which $T_0^{AB}=T_0=[4]$ and $T_0^{BA}=\emptyset$, and consider the 6-partite matching
   $$ABCDE\;ABCDE\;ABCDE\;EDABC\;|BACED\;CBADE.$$
   consisting of five edges $e_A,\dots,e_E$, corresponding to letters $A-E$.
   In the poset defined by~\eqref{po3}, the edges $e_A,e_B,e_C$ form a chain, while the edges $e_A,e_D,e_E$ form an anti-chain.

\subsection{Pattern avoiding permutations and proof of Proposition~\ref{fullhouse-1}}\label{pap}

For the  proof of Proposition~\ref{fullhouse-1} we will need a result of Gunby and P\'{a}lv\"{o}lgyi from \cite{GP2019} concerning parallel pattern avoidance by permutations.

For integers $k\ge m$, let $\sigma$ be a permutation of $[k]$ and let $\tau$ be a permutation of $[m]$ (called pattern). We say that $\sigma$ \emph{avoids} pattern $\tau$ if there is no sequence $a_1<\dots <a_m$ such that the sub-permutation  $\sigma(a_1),\dots, \sigma(a_m)$ is order-isomorphic to permutation $\tau$. This notion can be generalized in a natural way to sequences of permutations. Let $\sigma=(\sigma_1,\dots,\sigma_d)$ be a $d$-tuple of permutations of $[k]$ and let $\tau=(\tau_1,\dots,\tau_d)$ be a $d$-tuple of permutations of $[m]$ (patterns). We say that $\sigma$ \emph{avoids} $\tau$ if there is no sequence $a_1<\cdots <a_m$ such that   sub-permutation $\sigma_i(a_1),\dots, \sigma_i(a_m)$ is order-isomorphic to permutation $\tau_i$ for $i=1,\dots,d$. Given a $d$-tuple $\tau$ of permutations of $[m]$, let $f_d(k)$ denote the number of $d$-tuples of permutations of $[k]$ avoiding $\tau$.

Generalizing the celebrated result by Markus and Tardos \cite{MT}, Gunby and P\'{a}lv\"{o}lgyi proved the following estimates on $f_d(k)$.
\begin{thm}[\cite{GP2019}]\label{thm Gunby-Palvolgyi}
	Let $k\ge m$ and $\tau$ be a fixed $d$-tuple of permutations of $[m]$.  Then there exist positive constants $c_1$ and $c_2$ (depending on $\tau$) such that
	\[
	c_1^k\cdot k^{k(d-1/d)}\leqslant f_d(k)\leqslant c_2^k\cdot k^{k(d-1/d)}.
	\]
\end{thm}

  \begin{proof}[Proof of Proposition~\ref{fullhouse-1}]
 Let $\mathcal C(P,\pi)$ be a $t$-cube, where $P,\pi$ and all other notation is as in the proof of the upper bound in Theorem~\ref{fullhouse} presented in Section~\ref{generic}.  In addition, let $\pi'$ be a partition
 $$[s]=(T_0'\cup\cdots\cup T_u')\cup T_1\cup\cdots \cup T_t,$$
 where  $1\in T_0'$ and $T_1',\dots, T_u'$ are all non-empty, and $T_0=T_0'\cup\cdots\cup T_u'$. Further, set $P_j'$ for the concatenation of segments $S_i$ with $i\in T_j'$, and $U_j':=\bigcup_{i\in T_j'}V_i$, for $j=0,\dots,u$.
  If we were after the $(u+t)$-cube $\cC(P,\pi')$, we would, as
 before, just take a $(\sigma_1',\dots,\sigma_u',\sigma_1,\dots,\sigma_t)$-concatenation
 $$K[U'_0]\overset{\sigma_1'}{-}K[U_1']\overset{\sigma_2'}{-}\cdots\overset{\sigma_u'}{-}K[U_u']\overset{\sigma_1}{-} K[U_1]\overset{\sigma_2}{-}\cdots \overset{\sigma_t}{-} K[U_t].$$
  However, to avoid the forbidden $t$-cube $\cC(P,\pi)$,  the $u$-tuple of permutations $(\sigma'_1,\dots,\sigma'_u)$ should collectively avoid the $u$-tuple of permutations $(\tau_1,\dots,\tau_u)$ of $[2]:=\{1,2\}$ where $\tau_j=12$ if $P'_j$ begins with $A$ and $\tau_j=21$, otherwise.

Thus, by Theorem~\ref{thm Gunby-Palvolgyi} with $d=u$, we have $a_{\cP}(k)\le c_2^kk^{(u-1/u)k}\times (k!)^t$
for $\cP=\cC(P,\pi')\setminus  \cC(P,\pi)$. Consequently, by  Lemma~\ref{meta} with $C=c_2$ and $x=u-1/u+t$, we get, a.a.s., $z_{\cP}(\rm^{(r)}_{n})= O\left(n^{\frac{1}{r-u+1/u-t}}\right)$, which completes the proof.

\end{proof}

\begin{rem} For $P\in\cR^{(r)}$ and $t=0$, the bound on $a_{\cP}(k)$ in the above proof reduces to $a_{\cR^{(r)}\setminus\{P\}}(k)\le c_2^kk^{(r-1-1/(r-1))k}$ and could be used in the proof of Corollary~\ref{friends} instead of the result from \cite{AJKS}, but only in the special case when $\cP=\cR^{(r)}\setminus\{P\}$.
\end{rem}


\subsection{Reconstruction from traces and proof of Theorem~\ref{thm:random_two_patterns_gen} for mismatched pairs}\label{trace}
In some proofs of upper bounds, in order to estimate the number $a_{\cP}(k)$ of distinct $\cP$-cliques of size $k$ appearing in Lemma~\ref{meta}, one can employ the concept of the trace of a matching introduced in \cite{DGR-socks}.

Let $M$ be an ordered $r$-matching of size $n$. Assign number $i$ to the $i$-th (from the left) vertex of each edge of $M$. The obtained $r$-ary sequence (with alphabet $[r]$) will be called the \emph{trace of $M$} and denoted by $\tr(M)$. For instance,
for $M=AABCBDBDACCD$, we have $\tr(M)=121121323233$. In general, there may be many matchings with the same  trace, e.g.,  $M'=AABCCDADDBBC$ has $\tr(M')=\tr(M)$.
 However, for some restricted families of matchings $\cF$ it may happen that no two members of $\cF$ have the same trace. Then, knowing the trace of a matching $M\in\cF$, one is able to uniquely reconstruct $M$. In that case we call the family $\cF$ \emph{reconstructible}.
Using this concept, we obtain the following upper bound on the size of the largest $\cP$-clique in $\rm_n^{(r)}$ for a related class of pattern families $\cP$.


\begin{observation}\label{recon}
Let $\cP$ be a set of $r$-patterns such that the family of all $\cP$-cliques is reconstructible. Then, a.a.s.,  $z_{\cP}(\rm_n^{(r)})=O(n^{1/r})$.
\end{observation}
\begin{proof}
By assumption,  there are no more $\cP$-cliques on a vertex set of size $rk$ than there are words of length $rk$ over alphabet $[r]$. So, $a_{\cP}(k)\leqslant r^{rk}$ and the conclusion follows by Lemma \ref{meta} with $C=r^r$ and $x=0$.
\end{proof}

\begin{rem} Traces of ordered matchings have a special structure: for every pair $1\le i<j\le n$, in every prefix of the trace $\tr(M)$ the number of occurrences of digit $i$ is at least as large as the number of occurrences of digit $j$. Thus, they are $r$-dimensional Dyck words of length $rn$ defined in Section~\ref{generic} and enumerated by $r$-dimensional Catalan numbers $C_n^{(r)}$. In fact, one can show that $\cP_{Dyck}^{(r)}$-cliques of size $k$ are (bijectively) reconstructible, yielding the formula
$a_{\cP_{Dyck}^{(r)}}(k)=C_k^{(r)}$. Thus, recalling that, by definition, $a_{\cP_{Dyck}^{(r)}}(k)=C_r^{(k)}$,  we recover the identity $C_r^{(k)}=C_k^{(r)}$ used in the proof of Proposition~\ref{1-over-r}.
\end{rem}

In the remainder of this section we will use Observation~\ref{recon} to prove the upper bound in Theorem~\ref{thm:random_two_patterns_gen} for mismatched pairs (again, the lower bound follows by Theorem~\ref{thm:random_one_pattern} and  monotonicity).
But  first, let us make an observation about the inheritance  of the property of being mismatched.

For an $r$-matching $M$ and an integer $1\le j\le r$, let $M^{-j}$ denote the $(r-1)$-matching obtained  by removing the $j$-th letter from each edge of $M$. For instance, if $M=AABCBDBCCDDA$, then  $M^{-1}=ABBCCDDA$, $M^{-2}=ABCDBCDA$, and $M^{-3}=AABCBDCD$.
 In particular, this removal operation can be applied to $r$-patterns. For example,
 if $P=AABBBABA$, then $P^{-1}=P^{-2}=ABBABA$, while $P^{-3}=AABBBA$. Clearly, if $P$ is collectable, then so is $P^{-j}$ for every $j$, but the opposite does not necessarily hold.

\begin{observation}\label{obs:1}
For every $r\geq 3$ and any mismatched pair $\{P,Q\}$ of $r$-patterns, there exists $1\le j\le r$ such that $\{P^{-j},Q^{-j}\}$ is a mismatch, too.
\end{observation}

\proof We split the proof into two cases according to the mismatch definition.

{\bf Case 1:} Assume that $P$ and $Q$ are both collectable but do not form a harmonious pair.  Let $\lambda_P=(g_1,g_2,\dots)$ and $\lambda_Q=(h_1,h_2,\dots)$,
where $\lambda_P\neq\lambda_Q$.
If $g_1=h_1$, then take $j=1$, as, clearly, $\lambda_{P^{-1}}\neq\lambda_{Q^{-1}}$.
If (w.l.o.g.)  $g_1>h_1$, then take $j=r$, as $\lambda_{P^{-r}}\neq\lambda_{Q^{-r}}$,  unless $g_1=r$, $h_1=r-1$, and $h_2=1$, in which case  take $j=1$, as then $\lambda_{P^{-1}}\neq\lambda_{Q^{-1}}$ again.

\medskip

{\bf Case 2:} Assume that $P$ is collectable, but $Q$ is not. The case $r=3$ (and $Q=P_0=AABABB$) is the only one in which $Q^{-j}$ is collectable no matter what $j$ is (simply because all three 2-patterns are such).
Still, one can check by inspection that when choosing $j=3$ for each of the six 3-patterns beginning with $AB$, $j=1$ for $P=P_2$ and $P=P_3$, and $j=2$ for $P=P_1$, we obtain $\lambda_{P^{-j}}\neq\lambda_{Q^{-j}}$ in every case.

For $r\ge4$ we will show that, for some $j$,  $Q^{-j}$ remains non-collectable, which automatically makes the pair $P^{-j}, Q^{-j}$ a mismatch.
Being  non-collectable, $Q$ must be of the form $Q=RA^qB^tAS$ or $Q=RB^qA^tBS$ with $R$ collectable (possibly empty) and $q>t\ge 1$. By symmetry, we consider only the former case.  Note that then suffix $S$ contains at least two letters $B$.

If $R\neq\emptyset$, take $j=1$. Then $Q^{-1}=R'A^qB^tAS$  with $R'$ collectable (possibly empty), thus $Q^{-1}$ is non-collectable.

For $R=\emptyset$, consider three further subcases.
 If $t\ge2$, take $Q^{-1}=A^{q-1}B^{t-1}A S$, if $q\ge3$ and $t=1$, take $Q^{-2}=A^{q-1}BA S$, while if $q=2$ and $t=1$, take $Q^{-r}=AABS'$, where $S'$ is obtained from suffix $AS$ by removing the last letters $A$ and $B$.  In each case the resulting $(r-1)$-pattern $Q^{-j}$  is visibly  non-collectable.
 \qed

\medskip

\begin{proof}[Proof of Theorem~\ref{thm:random_two_patterns_gen}, mismatched pairs]
Let $\{P,Q\}$ be a mismatched pair of $r$-patterns.
In view of Observation~\ref{recon}, it is sufficient to demonstrate that the family of $\{P,Q\}$-cliques is reconstructible. We proceed by induction on $r$.  For $r=2$, it was already shown in the proof of Theorem~\ref{thm:random_two_patterns_r=2} above that for
both mismatched pairs,  $\mathcal P_1=\{AABB,\;ABAB\}$ and $\mathcal P_2=\{AABB,\;ABBA\}$, the corresponding families of cliques are reconstructible.

Assume now that $r\geq 3$ and that we have proved the statement for  uniformity $r-1$.
 Let $M$ and $N$ be two different $\{P,Q\}$-cliques of size $k$ on vertex set $V$. We are going to show that $\tr(M)\neq \tr(N)$. To this end, let  $j$ be given by Observation~\ref{obs:1}. After removing the $j$-th letter from each edge of $M$ and $N$, we obtain $\{P^{-j},Q^{-j}\}$-cliques $M^{-j}$ and $N^{-j}$, where  the pair $\{P^{-j},Q^{-j}\}$ is still a mismatch.
 If $V(M^{-j})\neq V(N^{-j})$, then $\tr(M)\neq \tr(N)$ as digit $j$ appears in $\tr(M)$ and $\tr(N)$ on different positions.

 Assume then that
 $V(M^{-j})= V(N^{-j})=V\setminus\{v_1,\dots,v_k\}$ where $v_1<\dots<v_k$. We are going to show that $M^{-j}\neq N^{-j}$ and then apply induction. Suppose that $M^{-j}= N^{-j}$.
  Let $e_1,\dots, e_k$ be the edges of $M$ and $e'_1,\dots, e'_k$ be the corresponding edges of $M^{-j}$. We may assume that $e_i=e_i'\cup\{v_i\}$, $i=1,\dots,k$. Then $N=\{e_i'\cup\{v_{\alpha(i)}\}: i=1,\dots,k\}$ where $\alpha$ is a permutation of $[k]$. Since $M\neq N$, $\alpha$ is not an identity and so there is an inversion in $\alpha$: a pair of indices $i<\ell$ such that $\alpha(i)>\alpha(\ell)$. W.l.o.g., let the pair $e_i,e_\ell$ form $r$-pattern $P$, so $e_i',e_\ell'$ form $P^{-j}$. However, due to the inversion, the edges $e_i'\cup\{v_{\alpha(i)}\}$ and $e_\ell'\cup\{v_{\alpha(\ell)}\}$ must form an $r$-pattern different from $P$ (as the order of the $j$-th $A$ and the $j$-th $B$ is reversed when compared to the pair $e_i,e_\ell$). So, they must form $Q$. This, however, contradicts the fact that $Q^{-j}\neq P^{-j}$.

 We thus proved that $M^{-j}\neq N^{-j}$, and consequently, by the induction hypothesis, $\tr(M^{-j})\neq \tr(N^{-j})$. Putting back digit $j$ to the original set of $k$ positions (the same for $M$ and for $N$), we infer that $\tr(M)\neq \tr(N)$. This concludes the proof of reconstructability of the family of $\{P,Q\}$-cliques, and thus completes the proof of Theorem~\ref{thm:random_two_patterns_gen} for mismatched pairs $\{P,Q\}$.

\end{proof}


\begin{rem}
In fact, the above proof does not need the notions of reconstructability. Indeed, based just on Observation~\ref{obs:1}, one could prove an exponential bound $a_{P,Q}(k)<2^{k\left\{\binom{r+1}2-1\right\}}$  directly, still by induction on $r$. For this, it suffices to show (as we did along the way) that the mapping assigning to a $\{P,Q\}$-clique $M$ of size $k$ an ordered pair $(\{v_1,\dots,v_k\},M^{-j})$ consisting of  the $j$-th vertices of the edges of $M$ and a $\{P^{-j},Q^{-j}\}$-clique $M^{-j}$ obtained from $M$ by removing them is injective. Then
$$a_{\{P,Q\}}(k)\le {rk\choose k}\cdot a_{\{P^{-j},Q^{-j}\}}(k)\leq 2^{rk}\cdot 2^{k\left\{\binom{r}2-1\right\}}=2^{k\left\{\binom{r+1}2-1\right\}}.$$
\end{rem}

\section{Pairs of non-collectable patterns}\label{non-collect}

For  non-collectable $r$-patterns $P$ and $Q$, the (a.a.s.) upper bound  $z_{\{P,Q\}}(\rm^{(r)}_{n})=O(1)$ in Theorem~\ref{thm:random_two_patterns_gen} follows easily by a Ramsey type argument. In fact, a stronger, fully deterministic,  statement holds. Let $\cP$ be a set of  non-collectable $r$-patterns  of size $|\cP|=m\ge2$. Then the largest $\cP$-clique has size smaller than the multicolor Ramsey number $R_{m}(3)$. Indeed, let $M$ be a $\cP$-clique with $k$ edges and consider the complete graph $K_k$ whose vertices are the edges of $M$ whereas each edge of $K_k$ is colored with the pattern formed by its ends. 
By \cite[Prop. 2.1]{JCTB_paper}, there is no monochromatic triangle in this $m$-colored clique, so $k<R_{m}(3)$ by the definition of Ramsey numbers. In particular, for $m=2$ we get the bound $k\le5$. It follows that, defining
$$z_{\cP}=\max\{k:\exists\;\;\mbox{$\cP$-clique of size $k$}\},$$
we have $z_{\cP}<R_{m}(3)$.
The fact that, a.a.s., $z_{\cP}(\rm^{(r)}_{n})=z_{\cP}$  stems from the following  more general result, applied to $H_n=H$ being the largest $\cP$-clique possible, that is, a $\cP$-clique of size $z_{\cP}$.

\begin{prop}\label{genH}
Let $H_n$ be a sequence of ordered $r$-matchings of size  $|H_n|\le \frac{1}{2r}\left(\frac{n}{\log n}\right)^{\frac{1}{r}}$. Then, a.a.s., there is a copy of $H_n$ in~$\rm^{(r)}_{n}$.
\end{prop}

For the proof we will need a lemma from \cite{JCTB_paper}. Recall that an edge $e$ spans sets $W_1,\dots,W_r$ if if $|e\cap W_i|=1$ for each $1\le i\le r$.
\begin{lemma}\label{lemma:span_prob} Given integers $r\ge2$, $t\ge2r$, and $n\ge t$,
let $W_1,\dots,W_r$ be disjoint subsets of $[rn]$ such that $|W_i|=t$, for each $1\le i \le r$. Then, the probability that no edge of $\rm^{(r)}_{n}$ spans all these sets is at most $\exp\left\{{-\frac{1}{(2r)^r} \cdot \frac{t^r}{n^{r-1}}}\right\}$. \qed
\end{lemma}

\begin{proof}[Proof of Proposition~\ref{genH}]
Let $t:=2r(\log n)^{\frac{1}{r}} n^{1-\frac{1}{r}}$. By adding some extra edges, we may assume that $H:=H_n$ is a matching of size \emph{exactly} $k:=n/t=\frac{1}{2r}\left(\frac{n}{\log n}\right)^{\frac{1}{r}}$, where, for simplicity, $t$ and $k$ are both deemed to be integers. We divide $[rn]$ into $rk$ consecutive blocks $B_1,\dots,B_{rk}$, each of length $t$. So,  $B_1=[t]$,  $B_2=\{t+1,\dots,2t\}$, etc.

We represent $H$ as a matching on vertex set $\{v_1,\dots,v_{rk}\}$, ordered consistently with the indices, disjoint from $[rn]$ and with edges $e_1,\dots,e_k$. 
For every $i=1,\dots,k$, let $I_i$ be the indicator random variable equal to 1 if there is \emph{no} edge in $\rm^{(r)}_{n}$ spanning the sets $B_j$, for all $j$ such that $v_j\in e_i$, and equal to 0 otherwise (see Figure \ref{K}). Further, let $X = \sum_{i=1}^k I_i$. We will show that, a.a.s., $X=0$, which implies that there is a copy of $H$ in $\rm^{(r)}_{n}$.

\begin{figure}[ht]
\captionsetup[subfigure]{labelformat=empty}
\begin{center}

\scalebox{1}
{
\centering
\begin{tikzpicture}
[line width = .5pt,
vtx/.style={circle,draw,black,very thick,fill=black, line width = 1pt, inner sep=0pt},
]

    \node[vtx] (1) at (0,0) {};
    \node[vtx] (2) at (1,0) {};
    \node[vtx] (3) at (2,0) {};
    \node[vtx] (4) at (3,0) {};
    \node[vtx] (5) at (4,0) {};
    \node[vtx] (6) at (5,0) {};
    \node[vtx] (7) at (6,0) {};
    \node[vtx] (8) at (7,0) {};
    \node[vtx] (9) at (8,0) {};
    \node[vtx] (10) at (8.8,0) {};

    \draw[line width=2.5mm, color=gray, line cap=round]  (1) -- (0.8,0) node[pos=0.5, below] {$B_1$};
    \draw[line width=2.5mm, color=gray, line cap=round]  (2) -- (1.8,0) node[pos=0.5, below] {$B_2$};
    \draw[line width=2.5mm, color=gray, line cap=round]  (3) -- (2.8,0) node[pos=0.5, below] {$B_3$};
    \draw[line width=2.5mm, color=gray, line cap=round]  (4) -- (3.8,0) node[pos=0.5, below] {$B_4$};
    \draw[line width=2.5mm, color=gray, line cap=round]  (5) -- (4.8,0) node[pos=0.5, below] {$B_5$};
    \draw[line width=2.5mm, color=gray, line cap=round]  (6) -- (5.8,0) node[pos=0.5, below] {$B_6$};
    \draw[line width=2.5mm, color=gray, line cap=round]  (7) -- (6.8,0) node[pos=0.5, below] {$B_7$};
    \draw[line width=2.5mm, color=gray, line cap=round]  (8) -- (7.8,0) node[pos=0.5, below] {$B_8$};
    \draw[line width=2.5mm, color=gray, line cap=round]  (9) -- (8.8,0) node[pos=0.5, below] {$B_9$};

    \draw[line width=0.3mm, color=lightgray]  (1) -- (10);
    \fill[fill=black, outer sep=-0.2mm]   (1) circle (0.1) node [below left] {$1$};
    \fill[fill=black, outer sep=-0.2mm]   (10) circle (0.1) node [below right] {$3n$};

    \node[vtx] (v1) at (0.75,0) {};
    \node[vtx] (v2) at (1.5,0) {};
    \node[vtx] (v3) at (2.5,0) {};
    \node[vtx] (v4) at (3.25,0) {};
    \node[vtx] (v5) at (4.3,0) {};
    \node[vtx] (v6) at (5.8,0) {};
    \node[vtx] (v7) at (6.4,0) {};
    \node[vtx] (v8) at (7.6,0) {};
    \node[vtx] (v9) at (8.5,0) {};
    \coordinate (v47) at (4.825,1.25);
    \coordinate (v36) at (4.15,1);
    \coordinate (v68) at (6.7,1);
    \coordinate (v25) at (2.9,1.12);
    \coordinate (v59) at (6.4,1.12);

    \node[color=black] at (-1.5,0) {$\rm^{(3)}_{n}\!\!:$};

    \draw[line width=0.5mm, color=blue, outer sep=2mm] (v4) arc (0:180:1.25);
    \draw[line width=0.5mm, color=blue, outer sep=2mm] plot [smooth, tension=2] coordinates {(v4) (v47) (v7)};
    \draw[line width=0.5mm, color=blue, outer sep=2mm] plot [smooth, tension=2] coordinates {(v3) (v36) (v6)};
    \draw[line width=0.5mm, color=blue, outer sep=2mm] plot [smooth, tension=2] coordinates {(v6) (v68) (v8)}; 
    \draw[line width=0.5mm, color=blue, outer sep=2mm] plot [smooth, tension=2] coordinates {(v2) (v25) (v5)}; 
    \draw[line width=0.5mm, color=blue, outer sep=2mm] plot [smooth, tension=2] coordinates {(v5) (v59) (v9)}; 

    \fill[fill=black, outer sep=1mm]   (v1) circle (0.1);
    \fill[fill=black, outer sep=1mm]   (v2) circle (0.1);
    \fill[fill=black, outer sep=1mm]   (v3) circle (0.1);
    \fill[fill=black, outer sep=1mm]   (v4) circle (0.1);
    \fill[fill=black, outer sep=1mm]   (v5) circle (0.1);
    \fill[fill=black, outer sep=1mm]   (v6) circle (0.1);
    \fill[fill=black, outer sep=1mm]   (v7) circle (0.1);
    \fill[fill=black, outer sep=1mm]   (v8) circle (0.1);
    \fill[fill=black, outer sep=1mm]   (v9) circle (0.1);

    \node[vtx] (k1) at (0.5,2.5) {};
    \node[vtx] (k2) at (1.5,2.5) {};
    \node[vtx] (k3) at (2.5,2.5) {};
    \node[vtx] (k4) at (3.5,2.5) {};
    \node[vtx] (k5) at (4.5,2.5) {};
    \node[vtx] (k6) at (5.5,2.5) {};
    \node[vtx] (k7) at (6.5,2.5) {};
    \node[vtx] (k8) at (7.5,2.5) {};
    \node[vtx] (k9) at (8.5,2.5) {};
    \coordinate (k25) at (3,3.75) {};
    \coordinate (k59) at (6.5,3.75) {};
    \coordinate (k45) at (4,3.5) {};
    \coordinate (k78) at (6.5,3.5) {};

    \draw[line width=0.3mm, color=lightgray]  (k1) -- (k9);

    \draw[line width=0.5mm, color=blue, outer sep=2mm] (k4) arc (0:180:1.5);
    \draw[line width=0.5mm, color=blue, outer sep=2mm] (k7) arc (0:180:1.5);
    \draw[line width=0.5mm, color=blue, outer sep=2mm] plot [smooth, tension=2] coordinates {(k2) (k25) (k5)};
    \draw[line width=0.5mm, color=blue, outer sep=2mm] plot [smooth, tension=2] coordinates {(k5) (k59) (k9)};
    \draw[line width=0.5mm, color=blue, outer sep=2mm] plot [smooth, tension=2] coordinates {(k3) (k45) (k6)};
    \draw[line width=0.5mm, color=blue, outer sep=2mm] plot [smooth, tension=2] coordinates {(k6) (k78) (k8)};

   \fill[fill=black, outer sep=1mm]   (k1) circle (0.1) node [below] {$v_1$};
   \fill[fill=black, outer sep=1mm]   (k2) circle (0.1) node [below] {$v_2$};
   \fill[fill=black, outer sep=1mm]   (k3) circle (0.1) node [below] {$v_3$};
   \fill[fill=black, outer sep=1mm]   (k4) circle (0.1) node [below] {$v_4$};
   \fill[fill=black, outer sep=1mm]   (k5) circle (0.1) node [below] {$v_5$};
   \fill[fill=black, outer sep=1mm]   (k6) circle (0.1) node [below] {$v_6$};
   \fill[fill=black, outer sep=1mm]   (k7) circle (0.1) node [below] {$v_7$};
   \fill[fill=black, outer sep=1mm]   (k8) circle (0.1) node [below] {$v_8$};
   \fill[fill=black, outer sep=1mm]   (k9) circle (0.1) node [below] {$v_9$};

    \node[color=black] at (-1.15,2.5) {$H\!\!:$};

\end{tikzpicture}
}

\end{center}

\caption{A matching $H$ of size 3 and its copy in $\rm^{(3)}_{n}$.}
\label{K}
			
\end{figure}
To this end, observe that for each $i$, by Lemma~\ref{lemma:span_prob} applied to the sets $B_j$, $v_j\in e_i$,	
\[
\PP(I_i=1)\le\exp\left\{-\frac{1}{(2r)^r} \cdot \frac{t^r}{n^{r-1}}\right\}=\exp\{-\log n\}.
\]
Finally, by Markov's inequality and the linearity of expectation,
\[
\PP(X\ge 1)\le
\E X = \sum_{i=1}^k \PP(I_i=1)\le n^{1/r}\exp\{-\log n\}=o(1),
\]
finishing the proof.	
\end{proof}

 One can get rid of the $\log n$ factor in Proposition~\ref{genH} if instead of an arbitrary matching~$H_n$ one is after a member of a hereditary family of matchings, a prominent example of which are $\mathcal P$-cliques. Indeed, by applying essentially the same easy  trick as in \cite[Remark 3.5]{JCTB_paper}, one can alter the final part of the proof of Proposition~\ref{genH} and get, a.a.s., $X\ge k/2$. This yields the following result.

\begin{prop}\label{herH} Let $\mathcal H$ be a family of matchings such that for every $k$ there is $H\in\mathcal H$ of size $|H|=k$ and whenever $H\in\mathcal H$, then for every $H'\subset H$ also $H'\in\mathcal H$.
 Then, a.a.s., $\rm^{(r)}_{n}$ contains a member of $\mathcal H$ of size $\Omega(n^{1/r}/\omega(n))$, where $\omega(n)\to\infty$ arbitrarily slowly. \qed
\end{prop}

In fact, one can even get rid of the $\omega(n)$ factor by applying Talagrand's inequality in the same way as in the proof of \cite[Theorem 1.8]{JCTB_paper}.




%


\section{ Cubes - the lower bound}\label{fhlb}
In this section we prove the lower bound in Theorem~\ref{fullhouse}. We begin with some preparations.

\subsection{Preparations}

Consider a collectable $r$-pattern $P$ with splitting $P=S_1\cdots S_s$ into blocks of sizes $|S_i|=2\lambda_i$, $i=1,\dots,s$.
Let $\pi$ be a partition $[s]=T_0\cup T_1\cup\cdots \cup T_t$, with $1\in T_0$ and all $T_j\neq\emptyset$, $j=1,\dots,t$.
The pair $(P,\pi)$ will be called the \emph{seed} of the $t$-cube $\cC(P,\pi)$. For $j=0,\dots,t$, set $w(T_j):=r_j:=\sum_{i\in T_j}\lambda_i$ and recall that the mega-block $P_j$ is the $r_j$-pattern formed by concatenation of segments $S_i$ with $i\in T_j$. As one may have $r_j=1$, we allow ``degenerate'' 1-patterns $AB$.

For $k\ge2$, define an auxiliary hypergraph $H_k(P,\pi)$ on $kr$ vertices as follows.
Let $V$ be a linearly ordered set of $kr$ vertices split into consecutive blocks $V_i$ of sizes $k\lambda_i$, that is, $V=\bigcup_{i=1}^s V_i$ and $|V_i|=k\lambda_i$, $i=1,\dots,s$. Further, for $j=0,\dots,t$, set $U_j=\bigcup_{i\in T_j}V_i$ and let $K_j$ be the (unique) $P_j$-clique of size $k$ on $U_j$ (for $r_j=1$, $K_j$ is a 1-matching consisting of $k$ isolated vertices marked by distinct letters). Finally,
$$H_k(P,\pi)=K_0\times K_1\times\cdots \times K_t,$$
where the product of matchings was defined in Section~\ref{defop}. Note that, equivalently,
$$H_k(P,\pi)=\bigcup\left\{K: \mbox{$K$ is a $\cC(P,\pi)$-clique on $[kr]$}\right\},$$
and that, by~\eqref{edges}, $|H_k(P,\pi)|=k^{t+1}$.

We will treat this hypergraph as \emph{ordered} or \emph{unordered}, depending on the context.
Note that, by definition, for every $Q\in \cC(P,\pi)$ we have $H_k(P,\pi)=H_k(Q,\pi)$.

One instance of $H_2(P,\pi)$ can be found in Example~\ref{product} in Section~\ref{defop}. Here is another one.

\begin{exmp}\label{E7}
For
$P=ABBA|ABAB$ and $Q=ABBA|BABA$, with $\lambda_P=(1,1,1,1)$ and $\pi=\{T_0,T_1\}$, where $T_0=\{1,2\}$ and $T_1=\{3,4\}$, we have $\cC(P,\pi)=\cC(Q,\pi)=\{P,Q\}$. Moreover, for $P$ we have $K_0=ABBA$ and $K_1=ABAB$, while for $Q$ we have $K_0=ABBA$ and $K_1=BABA$.
The hypergraph $H_3(P,\pi)=H_3(Q,\pi)$ on the vertex set
$V=U_0\cup U_1=V_1\cup V_2\cup V_3\cup V_4$, where $U_0=V_1\cup V_2$, $U_1=V_3\cup V_4$, and
$$V_1=\{1,2,3\},\quad V_2=\{4,5,6\},\quad V_3=\{7,8,9\},\quad V_4=\{10,11,12\},$$
is presented in Figure~\ref{fig:HPpi}. It has $3\times 3 =9$ edges, each of which is a union of an edge of $K_0$ and an
edge of $K_1$.
\end{exmp}

\begin{figure}[ht]
\captionsetup[subfigure]{labelformat=empty}
\begin{center}

\scalebox{1}
{
\centering
\begin{tikzpicture}
[line width = .5pt,
vtx/.style={circle,draw,red,very thick,fill=red, line width = 1pt, inner sep=0pt},
]

    \node[vtx] (1) at (0,0) {};
    \node[vtx] (13) at (11.8,0) {};

    \draw[line width=0.3mm, color=black]  (1) -- (13);

    \node[vtx] (1a) at (0+0.5,0) {};
    \node[vtx] (1b) at (1+0.5,0) {};
    \node[vtx] (1c) at (2+0.5,0) {};

    \node[vtx] (2a) at (3+0.5,0) {};
    \node[vtx] (2b) at (4+0.5,0) {};
    \node[vtx] (2c) at (5+0.5,0) {};

    \node[vtx] (3a) at (6+0.5,0) {};
    \node[vtx] (3b) at (7+0.5,0) {};
    \node[vtx] (3c) at (8+0.5,0) {};

    \node[vtx] (4a) at (9+0.5,0) {};
    \node[vtx] (4b) at (10+0.5,0) {};
    \node[vtx] (4c) at (11+0.5,0) {};

    \draw[line width=0.5mm, color=red, outer sep=2mm] (2c) arc (0:180:2.5);
    \draw[line width=0.5mm, color=red, outer sep=2mm] (2b) arc (0:180:1.5);
    \draw[line width=0.5mm, color=red, outer sep=2mm] (2a) arc (0:180:0.5);

    \draw[line width=0.5mm, color=red, outer sep=2mm] (4c) arc (0:180:1.5);
    \draw[line width=0.5mm, color=red, outer sep=2mm] (4b) arc (0:180:1.5);
    \draw[line width=0.5mm, color=red, outer sep=2mm] (4a) arc (0:180:1.5);

    \draw[line width=0.5mm, color=blue, outer sep=2mm] (3a) arc (0:180:0.5);
    \draw[line width=0.5mm, color=blue, outer sep=2mm] (3a) arc (0:180:1.0);
    \draw[line width=0.5mm, color=blue, outer sep=2mm] (3a) arc (0:180:1.5);

    \draw[line width=0.5mm, color=blue, outer sep=2mm] (3b) arc (0:180:1.0);
    \draw[line width=0.5mm, color=blue, outer sep=2mm] (3b) arc (0:180:1.5);
    \draw[line width=0.5mm, color=blue, outer sep=2mm] (3b) arc (0:180:2.0);

    \draw[line width=0.5mm, color=blue, outer sep=2mm] (3c) arc (0:180:1.5);
    \draw[line width=0.5mm, color=blue, outer sep=2mm] (3c) arc (0:180:2.0);
    \draw[line width=0.5mm, color=blue, outer sep=2mm] (3c) arc (0:180:2.5);

    \fill[fill=black, outer sep=1mm]   (1a) circle (0.1);
    \fill[fill=black, outer sep=1mm]   (1b) circle (0.1);
    \fill[fill=black, outer sep=1mm]   (1c) circle (0.1);

    \fill[fill=black, outer sep=1mm]   (2a) circle (0.1);
    \fill[fill=black, outer sep=1mm]   (2b) circle (0.1);
    \fill[fill=black, outer sep=1mm]   (2c) circle (0.1);

    \fill[fill=black, outer sep=1mm]   (3a) circle (0.1);
    \fill[fill=black, outer sep=1mm]   (3b) circle (0.1);
    \fill[fill=black, outer sep=1mm]   (3c) circle (0.1);

    \fill[fill=black, outer sep=1mm]   (4a) circle (0.1);
    \fill[fill=black, outer sep=1mm]   (4b) circle (0.1);
    \fill[fill=black, outer sep=1mm]   (4c) circle (0.1);

    \node[color=black] at (0+0.5,-0.5) {$1$};
    \node[color=black] at (1+0.5,-0.5) {$2$};
    \node[color=black] at (2+0.5,-0.5) {$3$};
    \node[color=black] at (3+0.5,-0.5) {$4$};
    \node[color=black] at (4+0.5,-0.5) {$5$};
    \node[color=black] at (5+0.5,-0.5) {$6$};
    \node[color=black] at (6+0.5,-0.5) {$7$};
    \node[color=black] at (7+0.5,-0.5) {$8$};
    \node[color=black] at (8+0.5,-0.5) {$9$};
    \node[color=black] at (9+0.5,-0.5) {$10$};
    \node[color=black] at (10+0.5,-0.5) {$11$};
    \node[color=black] at (11+0.5,-0.5) {$12$};

    \node[color=red] at (0+0.5,2.0) {$K_0$};
    \node[color=red] at (9+0.5,2.0) {$K_1$};

\end{tikzpicture}
}

\end{center}

\caption{An exemplary $H_3(P,\pi)$ with $P=ABBAABAB$  and $\pi=\{\{1,2\},\{3,4\}\}$. Each of its nine $4$-vertex edges is depicted by a red-blue-red path.}
\label{fig:HPpi}
			
\end{figure}	

In fact, the hypergraph $H_k(P,\pi)$ is much more universal.

 \begin{observation}\label{universal}
 Let $P$ and $P'$ be two collectable $r$-patterns with $\lambda_{P}=s$ and $\lambda_{P'}=s'$, and let $\pi$ and $\pi'$ be, respectively, partitions $[s]=T_0\cup \dots\cup T_t$ and $[s']=T_0'\cup \dots\cup T_t'$ into the \emph{same} number of sets and such that $1\in T_0\cap T_0'$ and $w(T_j')=w(T_j)$ ($=r_j$), for all $j=0,\dots,t$. Then  $H_k(P,\pi)$ and $H_k(P',\pi')$, are isomorphic (but not necessarily \emph{order} isomorphic). \qed
 \end{observation}

  This observation is quite obvious. Just recall that both, $H_k(P,\pi)$  and $H_k(P',\pi')$ are products of disjoint $r_j$-matchings of size $k$, $j=0,\dots,t$. So, it suffices to independently define isomorphisms between $K_j$ and $K_j'$, $j=0,\dots,t$.

\begin{exmp}\label{E8}
Let us illustrate Observation~\ref{universal} by taking the seed $(P,\pi)$ and in Example~\ref{E7} and a new seed $(P',\pi')$, where $P'=AABBBAAB$, $T'_0=\{1\}$, and $T'_1=\{2,3\}$ (note that $s'=3$). We have $K_0'=AABB$ and $K_1'=BAAB$. Hypergraph $H_k(P,\pi)$ is presented in Figure~\ref{fig:HPpi}, while $H_k(P',\pi')$ in Figure~\ref{fig:HPpi_prime}, with
  $V_1'=\{1',\dots,6'\}$, $V_2'=\{7',8',9'\}$, and $V_3'=\{10',11'12'\}$. They are visibly isomorphic (though, not order isomorphic) to each other. As one of many existing isomorphisms,  map $V_1$ onto $\{1',3',5'\}$ (in any order), and $V_4$ onto $\{10',11',12'\}$ (in any order).
   Then, the mappings between $V_2$ and $\{2',4',6'\}$ and between $V_3$ and $\{8',9',10'\}$ are uniquely
    determined.
\end{exmp}

\begin{figure}[ht]
\captionsetup[subfigure]{labelformat=empty}
\begin{center}

\scalebox{1}
{
\centering
\begin{tikzpicture}
[line width = .5pt,
vtx/.style={circle,draw,red,very thick,fill=red, line width = 1pt, inner sep=0pt},
]

    \node[vtx] (1) at (0,0) {};
    \node[vtx] (13) at (11.8,0) {};

    \draw[line width=0.3mm, color=black]  (1) -- (13);

    \node[vtx] (1a) at (0+0.5,0) {};
    \node[vtx] (1b) at (1+0.5,0) {};
    \node[vtx] (1c) at (2+0.5,0) {};

    \node[vtx] (2a) at (3+0.5,0) {};
    \node[vtx] (2b) at (4+0.5,0) {};
    \node[vtx] (2c) at (5+0.5,0) {};

    \node[vtx] (3a) at (6+0.5,0) {};
    \node[vtx] (3b) at (7+0.5,0) {};
    \node[vtx] (3c) at (8+0.5,0) {};

    \node[vtx] (4a) at (9+0.5,0) {};
    \node[vtx] (4b) at (10+0.5,0) {};
    \node[vtx] (4c) at (11+0.5,0) {};

    \draw[line width=0.5mm, color=red, outer sep=2mm] (2c) arc (0:180:0.5);
    \draw[line width=0.5mm, color=red, outer sep=2mm] (2a) arc (0:180:0.5);
    \draw[line width=0.5mm, color=red, outer sep=2mm] (1b) arc (0:180:0.5);

    \draw[line width=0.5mm, color=red, outer sep=2mm] (4c) arc (0:180:2.5);
    \draw[line width=0.5mm, color=red, outer sep=2mm] (4b) arc (0:180:1.5);
    \draw[line width=0.5mm, color=red, outer sep=2mm] (4a) arc (0:180:0.5);

    \draw[line width=0.5mm, color=blue, outer sep=2mm] (3a) arc (0:180:2.5);
    \draw[line width=0.5mm, color=blue, outer sep=2mm] (3b) arc (0:180:3);
    \draw[line width=0.5mm, color=blue, outer sep=2mm] (3c) arc (0:180:3.5);

    \draw[line width=0.5mm, color=blue, outer sep=2mm] (3a) arc (0:180:1.5);
    \draw[line width=0.5mm, color=blue, outer sep=2mm] (3b) arc (0:180:2);
    \draw[line width=0.5mm, color=blue, outer sep=2mm] (3c) arc (0:180:2.5);

    \draw[line width=0.5mm, color=blue, outer sep=2mm] (3a) arc (0:180:0.5);
    \draw[line width=0.5mm, color=blue, outer sep=2mm] (3b) arc (0:180:1);
    \draw[line width=0.5mm, color=blue, outer sep=2mm] (3c) arc (0:180:1.5);

    \fill[fill=black, outer sep=1mm]   (1a) circle (0.1);
    \fill[fill=black, outer sep=1mm]   (1b) circle (0.1);
    \fill[fill=black, outer sep=1mm]   (1c) circle (0.1);

    \fill[fill=black, outer sep=1mm]   (2a) circle (0.1);
    \fill[fill=black, outer sep=1mm]   (2b) circle (0.1);
    \fill[fill=black, outer sep=1mm]   (2c) circle (0.1);

    \fill[fill=black, outer sep=1mm]   (3a) circle (0.1);
    \fill[fill=black, outer sep=1mm]   (3b) circle (0.1);
    \fill[fill=black, outer sep=1mm]   (3c) circle (0.1);

    \fill[fill=black, outer sep=1mm]   (4a) circle (0.1);
    \fill[fill=black, outer sep=1mm]   (4b) circle (0.1);
    \fill[fill=black, outer sep=1mm]   (4c) circle (0.1);

    \node[color=black] at (0+0.5,-0.5) {$1'$};
    \node[color=black] at (1+0.5,-0.5) {$2'$};
    \node[color=black] at (2+0.5,-0.5) {$3'$};
    \node[color=black] at (3+0.5,-0.5) {$4'$};
    \node[color=black] at (4+0.5,-0.5) {$5'$};
    \node[color=black] at (5+0.5,-0.5) {$6'$};
    \node[color=black] at (6+0.5,-0.5) {$7'$};
    \node[color=black] at (7+0.5,-0.5) {$8'$};
    \node[color=black] at (8+0.5,-0.5) {$9'$};
    \node[color=black] at (9+0.5,-0.5) {$10'$};
    \node[color=black] at (10+0.5,-0.5) {$11'$};
    \node[color=black] at (11+0.5,-0.5) {$12'$};

    \node[color=red] at (0+0.5,1.0) {$K_0$};
    \node[color=red] at (11+0.5,2.0) {$K_1$};

\end{tikzpicture}
}

\end{center}

\caption{An exemplary $H_3(P',\pi')$ from Example~\ref{E8}, with $P'=AABBBAAB$  and $\pi'=\{\{1\},\{2,3\}\}$, isomorphic to $H_3(P,\pi)$ in Figure~\ref{fig:HPpi}.
}
\label{fig:HPpi_prime}
			
\end{figure}

Another important property of the auxiliary hypergraph $H_k(P,\pi)$ is stated in the next observation.

\begin{observation}\label{obs:MatchingInHkPpi} For all $P$ and $\pi$ as above,
every two disjoint edges of $H_k(P,\pi)$ form a pattern belonging to the cube $\cC(P,\pi)$.
\end{observation}

\begin{proof}
Consider two disjoint edges $e_1,e_2$ of $H_k(P,\pi)$ and, for $h=1,2$, write $e_h=e_h^{(0)}\cup e_h^{(1)}\cup\ldots\cup e_h^{(t)}$, where $e_h^{(j)}\in K_j$, $j=0,\dots,t$.
Let us focus on the (linearly ordered) vertex set $e_1\cup e_2$ of size $2r$ and replace every vertex of $e_1$ by $A$ and every vertex of $e_2$ by $B$, obtaining some pattern  $Q$.
 Then, for each $j=0,\dots,t$, the $A$'s and the $B$'s of
$Q$ corresponding to the positions in $e^{(j)}_1 \cup e^{(j)}_2$
 form a pattern $P_j$ or its flip  (depending on the relative positions of $e^{(j)}_1$ and $e^{(j)}_2$). Thus, by the definition of $\cC(P,\pi)$, we have $Q\in \cC(P,\pi)$. \end{proof}

\begin{exmp}\label{E7-cont}
Consider  edges $\{1,6,8,11\},\{2,5,7,10\}$, and $\{2,5,9,12\}$ from the hypergraph $H(P,\pi)$ presented in Example~\ref{E7} (see Figure~\ref{fig:HPpi}). The first two form pattern $ABBABABA$, while the first and the last - pattern $ABBAABAB$, both patterns belonging to $\cC(P,\pi)$.
\end{exmp}

\medskip

In the forthcoming proof we are going to utilize the notion of a blow-up of a hypergraph. For a hypergraph $H$ and an integer $\ell\ge2$,  the \emph{$\ell$-blow-up} of $H$, denoted $H^{(\ell)}$, is the  hypergraph obtained from $H$ by replacing every vertex $u$ by a set $W_u$, with $|W_u|=\ell$, where all sets $W_u$ are pairwise disjoint, and replacing every edge $\{u_{i_1},\dots,u_{i_r}\}$ of $H$ by the edges of the \emph{complete} $r$-partite hypergraph with vertex partition sets $W_{u_{i_1}},\dots, W_{u_{i_r}}$. The set $W_u$ will be referred to as a \emph{$\ell$-blow-up of vertex $u$}.
Moreover, we assume that the {\em ordered} hypergraph $H^{(\ell)}$ inherits the vertex ordering from
$H$. That is, for any vertices $u,w$ of $H$, if $u<w$, then all vertices of the blow-up set $W_u$ precede all vertices in $W_w$ (we will use notation $W_u<W_w$ for that).

Note that if hypergraphs $H$ and $H'$ are isomorphic, then so are their $\ell$-blow-ups, $H^{(\ell)}$ and $(H')^{(\ell)}$. What is more, there is an isomorphism (call it \emph{elegant})  which preserves the blow-up sets. This means that if $\phi$ is an isomorphism between $H$ and $H'$, then there is an (elegant) isomorphism  $\phi^{(\ell)}$    between $H^{(\ell)}$ and $(H')^{(\ell)}$ such that for every $u'\in W_u$, we have $\phi^{(\ell)}(u')\in W_{\phi(u)}$. The existence of elegant isomorphisms will prove to be important soon.

Before turning to the proof of Theorem~\ref{fullhouse}, we need to define one crucial notion.
A matching  $M$ in an $\ell$-blow-up $H^{(\ell)}$ is called \emph{scattered} if $|V(M)\cap W_u|\le1$ for each $u\in V(H)$.
Note that  if $H$ and $H'$ are isomorphic, then any elegant isomorphism of $H^{(\ell)}$ and $(H')^{(\ell)}$ preserves scattered matchings.

 Note also that, owing to the vertex ordering preservation, any scattered matching in $H^{(\ell)}$ is order isomorphic to a matching in $H$.
By Observation~\ref{obs:MatchingInHkPpi}, we thus immediately obtain the following.
\begin{observation}\label{obs:MatchingInHkPpiXnu}
Every scattered matching in $H_k(P,\pi)^{(\ell)}$ is a $\cC(P,\pi)$-clique. \qed
\end{observation}

For an example of a scattered matching see Figure~\ref{fig:HQ4} below.

\subsection{Proof of Theorem~\ref{fullhouse}, lower bound}

We first outline the idea of proof. Let $P$ and $\pi$ be as in the assumptions of Theorem~\ref{fullhouse}.
For $n$ sufficiently large,  set $k:=\lfloor n^{\frac{1}{r-t}}\rfloor$ and  $\ell:=\lfloor n^{1-\frac{1}{r-t}}\rfloor$, and let $H$ be an ordered copy of the  $\ell$-blow-up $H_k(P,\pi)^{(\ell)}$
 planted on the vertex set $[rk\ell]$ (with the natural linear ordering). Let a random variable $X$ be the largest size of a scattered matching in $\rm_n^{(r)}\cap H$. By Observation~\ref{obs:MatchingInHkPpiXnu}, $z_{\cC(P,\pi)}(\rm^{(r)}_{n})\ge X$.
 We will prove that $\E X\ge c_r k$, and then apply  Lemma~\ref{tala} to show sharp concentration of $X$ around its median $m$, and, ultimately, around $\E X$, which will complete the proof.

 Although, we could proceed directly with that plan, the estimates and, most of all, notation, would be quite cumbersome. Therefore, we resort  to a little trick allowing for ``switching horses'' and achieving the same goal in a more economic way. Namely, by Observation~\ref{universal}, the hypergraph $H_k(P,\pi)^{(\ell)}$ is elegantly isomorphic to  (simpler to analyze) $H_k((AB)^r,\pi')$  with the seed consisting of the $r$-partite pattern $(AB)^r$, called \emph{$r$-wave}, and the partition $\pi'$  of $[r]$ into \emph{consecutive} blocks of sizes $r_j$, $j=0,\dots,t$. Note that each mega-block $P_j$ is an $r_j$-wave itself.

 Let $H'$ be an ordered copy of the  $\ell$-blow-up $H_k((AB)^r,\pi')^{(\ell)}$
 planted on the vertex set $[rk\ell]$ (with the natural linear ordering) and let $X'$ be defined analogously to $X$, but with respect to $H'$. Owing to the  isomorphism, the probability spaces of random sub-matchings $\rm_n^{(r)}\cap H$ and $\rm_n^{(r)}\cap H'$ are stochastically equivalent, and, in turn, since there is an elegant isomorphism, the distributions of $X$ and $X'$ are the same. Consequently, it suffices to show  instead that $\E X'\ge c_r k$.
Before submerging into the details,  let us see an example alluring to Example~\ref{E7}.

\begin{exmp}\label{AB4} Let $k=3$, $P$ and $\pi$ be as in Example~\ref{E7}, $P'=(AB)^4$,
and $\pi'=\{1,2\}\cup\{3,4\}$. Hypergraph $H_3(P',\pi')$, presented in Figure~\ref{fig:HQ4}, is isomorphic to $H_3(P,\pi)$ in Figure~\ref{fig:HPpi}.
As this figure is supposed to illustrate  some steps in the forthcoming proof, we denote the vertices by $w_h^i$, $i=1,2,3,4$, $h=1,2,3$, and also indicate how a 3-blow-up $H_3^{(3)}(P',\pi')$ would look like.
\end{exmp}

\begin{figure}[ht]
\captionsetup[subfigure]{labelformat=empty}

\scalebox{1}
{
\centering
\begin{tikzpicture}
[line width = .5pt,
vtx/.style={circle,draw,black,very thick,fill=black, line width = 0pt, inner sep=0pt},
]

	\foreach \i in {1,...,12}
	{	
		\node[vtx] (\i) at (\i*1,0) {};
	}
    	\node[vtx] (0) at (0.5,0) {};
    	\node[vtx] (13) at (12.5,0) {};
	
	\foreach \i in {4,5,6,10,11,12}
	{
		\draw[line width=0.5mm, color=red, outer sep=2mm] (\i) arc (0:180:1.5);
	}
	
	\draw[line width=0.5mm, color=blue, outer sep=2mm] (7) arc (0:180:0.5);
	\draw[line width=0.5mm, color=blue, outer sep=2mm] (7) arc (0:180:1);
	\draw[line width=0.5mm, color=blue, outer sep=2mm] (7) arc (0:180:1.5);
	
	\draw[line width=0.5mm, color=blue, outer sep=2mm] (8) arc (0:180:1);
	\draw[line width=0.5mm, color=blue, outer sep=2mm] (8) arc (0:180:1.5);
	\draw[line width=0.5mm, color=blue, outer sep=2mm] (8) arc (0:180:2);
	
	\draw[line width=0.5mm, color=blue, outer sep=2mm] (9) arc (0:180:1.5);
	\draw[line width=0.5mm, color=blue, outer sep=2mm] (9) arc (0:180:2);
	\draw[line width=0.5mm, color=blue, outer sep=2mm] (9) arc (0:180:2.5);

%
%
%

	\foreach \i in {1,...,4}
	{	
		\foreach \j in {1,...,3}
		{
		        \pgfmathsetmacro{\result}{3*(\i-1)+\j}
		
			\fill[fill=black, outer sep=0.75mm]  (\result) circle (0.075) node [above right=-0.17cm] {{\footnotesize{$\,w^{\i}_{\j}$}}};

			\fill[fill=black, outer sep=0.75mm]  (\result) ++(0,-1.5) circle (0.06);
			\fill[fill=black, outer sep=0.75mm]  (\result) ++(0,-2) circle (0.06);
			\fill[fill=black, outer sep=0.75mm]  (\result) ++(0,-2.5) circle (0.06);

			\draw[line width=0.3mm] (\result) ++(0,-2) ellipse (0.25cm and 0.7cm) node [above=0.75cm, right=0cm] {{\footnotesize{$W^{\i}_{\j}$}}};
                		\draw[line width=0.3mm,<-] (\result)++(0,-0.25) -- ++(0,-0.85);
			
		}
	}
	
	\draw[line width=0.5mm, dotted] (1) ++(0,-1.5)  .. controls (2.5,-4) ..  (4,-2.5);
	\draw[line width=0.5mm, dotted] (4) ++(0,-2.5)  .. controls (6.75,-3.9) ..  (9,-2);
	\draw[line width=0.5mm, dotted] (9) ++(0,-2)  .. controls (10.5,-4) ..  (12,-2);
	
	\fill[fill=black, outer sep=0.75mm]  (1) ++(0,-1.5) circle (0.1);
	\fill[fill=black, outer sep=0.75mm]  (4) ++(0,-2.5) circle (0.1);
	\fill[fill=black, outer sep=0.75mm]  (9) ++(0,-2) circle (0.1);
	\fill[fill=black, outer sep=0.75mm]  (12) ++(0,-2) circle (0.1);
	
	\draw[line width=0.5mm, dotted] (2) ++(0,-2.5)  .. controls (3.5,-4) ..  (5,-2);
	\draw[line width=0.5mm, dotted] (5) ++(0,-2)  .. controls (6.75,-3.75) ..  (8,-1.5);
	\draw[line width=0.5mm, dotted] (8) ++(0,-1.5)  .. controls (9.5,-4) ..  (11,-2.5);
	
	\fill[fill=black, outer sep=0.75mm]  (2) ++(0,-2.5) circle (0.1);
	\fill[fill=black, outer sep=0.75mm]  (5) ++(0,-2) circle (0.1);
	\fill[fill=black, outer sep=0.75mm]  (8) ++(0,-1.5) circle (0.1);
	\fill[fill=black, outer sep=0.75mm]  (11) ++(0,-2.5) circle (0.1);
	
	\draw[line width=0.3mm, color=black]  (0) -- (13);


\end{tikzpicture}
}

\caption{The 4-uniform hypergraph $H_3((AB)^{4},\pi')$ described in Example~\ref{AB4} and its 3-blow-up. Two  good edges are drawn in dotted lines; they are of types $(1,3)$ and $(2,2)$, so they are separated and form a scattered matching of size two.
}
\label{fig:HQ4}
			
\end{figure}	


\begin{proof} [Proof of Theorem~\ref{fullhouse}]

Since we have switched from $H_k(P,\pi)^{(\ell)}$ to $H_k((AB)^r,\pi')^{(\ell)}$ and will work exclusively with the latter, we suppress the prime signs in $\pi'$, $H'$, and $X'$.
Let $r\ge 2$, $1\le t\le r-1$, and $n$ sufficiently large. For clarity of presentation,  assume that $k=n^{\frac{1}{r-t}}$ and $\ell = n^{1-\frac{1}{r-t}}$ are integers (the general case brings only additional multiplicative factors  $(1-o(1))$ in the estimates, which have a negligible effect.)

Set
$Q=(AB)^r$ for convenience and let $\pi$ be a partition of the form $[r]=T_0\cup T_1\cup\cdots\cup T_t$ where the $T_i$'s are blocks of consecutive integers of length $r_i$, that is, $T_0=[r_0]$, $T_1=\{r_0+1,\dots,r_0+r_1\}$, and so on. 
 Further, let $V$ be a linearly ordered set of vertices of $H_k(Q,\pi)$, $V=V_1\cup\cdots \cup V_r$, where $V_i=w_1^i,w_2^i,\dots, w_k^i$, $i=1,\dots,r$, (cf.\ Figure~\ref{fig:HQ4}).

Also, for $j=0,\dots,t$,
let $Q_j=(AB)^{r_j}$ be the $j$-th mega-block of $Q$
defined by partition $\pi$, and  set $K_j$ for the unique $Q_j$-clique of size $k$ on the vertex set
$$U_j=\bigcup_{i\in T_i}V_i=\{w_1^{r_0+\cdots +r_{j-1}+1},w_2^{r_0+\cdots +r_{j-1}+1},\ldots,w_k^{r_0+\cdots +r_j}\}.$$
(In Figure~\ref{fig:HQ4}, $U_0=\{w_1^1,w_2^1,w_3^1, w_1^2,w_2^2,w_3^2\}$ and $U_1=\{w_1^3,w_2^3,w_3^3, w_1^4,w_2^4,w_3^4\}$.)

Suppose a copy $H$ of the ordered blow-up $H_k(Q,\pi)^{(\ell)}$ of $H_k(Q,\pi)$
is fixed on the vertex set $[rk\ell]$
 and let
\[
W_1^1<W_2^1<\cdots< W_k^1<W_1^2<W_2^2<\cdots< W_k^2<\cdots<W_1^r<W_2^r<\cdots< W_k^r
\] be the corresponding blow-up sets (each of size~$\ell$), where $U<W$ means that every vertex of $U$ is to the left of every vertex of $W$.

For an edge $e$ of $\rm_n^{(r)}$ to belong to a scattered matching of $H$, it is necessary that

\begin{itemize}
\item $e$ contains at most one vertex from each set $W^i_h$ and, most importantly,
\item within the blow-up of each $Q_j$-clique $K_j$, the $r_j$ sets $W^i_h$ which contribute vertices to $e$ have the same lower index (because $Q_j$ is an $r_j$-wave).
    \end{itemize}
Consider, for example, the blow-up  in  Figure~\ref{fig:HQ4}. Here a good edge should take, for some $h_0,h_1\in[3]$, one vertex from both sets $W^1_{h_0},W^2_{h_0}$
 and one vertex from both sets  $W^3_{h_1},W^4_{h_1}$.

Generally, an edge $e\in\rm_n^{(r)}$ is  deemed \emph{good} if, for some vector $\omega=(h_0,h_1,\ldots,h_t)\in[k]^{t+1}$,  $e$ contains one vertex from each of the sets
\[
W_{h_0}^1,W_{h_0}^2,\dots,W_{h_0}^{r_0}, W_{h_1}^{r_0+1},W_{h_1}^{r_0+2},\dots,W_{h_1}^{r_0+r_1},\dots, W_{h_t}^{r_0+\dots +r_{t-1}+1},W_{h_t}^{r_0+\dots+ r_{t-1}+2},\dots,W_{h_t}^{r}.
\]
We then say  that such an edge is of \emph{type} $\omega$. The good edges in Figure~\ref{fig:HQ4} are of types $(1,3)$ and $(2,2)$.

A pair of vectors $\omega=(h_0,h_1,\ldots,h_t)$ and $\omega'=(h'_0,h'_1,\ldots,h'_t)$ is called \emph{totally distinct} if $h_j\neq h'_j$ for all $j=0,\dots,t$.
A pair of good edges $e,e'$ of types $\omega$ and $\omega'$, respectively, is called \emph{separated}  if the pair $(\omega,\omega')$ is totally distinct.
Note that any set of pairwise separated good edges forms a scattered $H_k(Q,\pi)^{(\ell)}$-matching.

In order to bound the expectation of $X$ from below we use the deletion method. To this end, let $Y$ be the number of good edges in~$\rm_n^{(r)}$. Since there are $k^{t+1}$ choices of a vector $\omega=(h_0,h_1,\ldots,h_t)$ and $|W_h^j|=\ell$ for all $h$ and $j$, we have, by~\eqref{1edge},
\[
\E Y = k^{t+1} \ell^r \cdot \frac{1}{\binom{rn-1}{r-1}}
\ge k^{t+1} \ell^r \frac{(r-1)!}{(rn)^{r-1}}
= kn^{\frac{t}{r-t}} n^{r-\frac{r}{r-t}} \frac{(r-1)!}{(rn)^{r-1}}
= \frac{(r-1)!}{r^{r-1}} k.
\]

Now, let $Z$  count  non-separated pairs $\{e,e'\}$ of good edges. Let $\omega=(h_0,h_1,\ldots,h_t)$ and $\omega'=(h_0',h_1',\ldots,h_t')$ be the types of $e$ and $e'$, respectively. Since $e$ and $e'$ are not separated, $\omega$ and $\omega'$ are not totally distinct and, consequently, there must be at least one index $j\in\{0,\dots,t\}$ such that $h_j=h_j'$. Hence, for a fixed~$e$ the number of choices of $e'$ is at most $(t+1)k^t \ell^r\le rk^t \ell^r$ and thus
\[
\E Z \le \frac{1}{2} k^{t+1} \ell^r r k^{t} \ell^r \cdot \frac{1}{\binom{rn-1}{r-1}\binom{rn-r-1}{r-1}}
= \E Y \cdot rn^{r-1} \cdot \frac{1}{2\binom{rn-r-1}{r-1}},
\]
where the constant $\frac12$ takes into account that the pairs $\{e,e'\}$ are unordered. We further bound the denominator as follows:

\begin{align*}
2\binom{rn-r-1}{r-1}&=\frac{2(rn)^{r-1}}{(r-1)!}\left(1-\frac{r+1}{rn}\right)\cdots\left(1-\frac{2r-1}{rn}\right)
\ge\frac{2(rn)^{r-1}}{(r-1)!}\left(1-\frac2n\right)^{r-1}\\&\ge\frac{2(rn)^{r-1}}{(r-1)!}\left(1-\frac{2r}n\right)\ge \frac{3(rn)^{r-1}}{2(r-1)!},
\end{align*}
where the penultimate inequality comes from Bernoulli's inequality and the last one holds for $n\ge8r$. Hence,
$\E Z\le \frac{2r!}{3r^{r-1}}\cdot\E Y  $. (Note that $\frac{2r!}{3r^{r-1}}<1$ for all $r\ge2$.)


After removing form the set of all $Y$  good edges one edge from each non-separated pair, we get a set of good, pairwise separated edges, and thus, a scattered matching of size at least $Y-Z$. Therefore, $X\ge Y-Z$ and

\begin{align}\label{EXbound}
\E X \ge \E(Y-Z) = \E Y - \E Z &\ge \E Y \left(1 - \frac{2r!}{3r^{r-1}}\right) \notag \\
&\ge \frac{(r-1)!}{r^{r-1}} \left(1 - \frac{2r!}{3r^{r-1}}\right) k =
\Omega_r(k).
\end{align}


It remains to apply Lemma~\ref{tala} to show sharp concentration of $X$ around its median $m$, and, ultimately, around $\E X$. Recall (cf.\ Section~\ref{section:random}) that $\rm^{(r)}_{n}$ can be generated by using the permutational scheme. Thus, one can view $X$ as a function $h$ of an $N$-permutation $\Pi_{N}$ with $N=rn$, that is, $X=h({\Pi}_{N})$. Notice that swapping two elements of $\Pi_{N}$ affects at most two edges of $\rm^{(r)}_{n}$, so it changes the value of $h$ by at most~2. Moreover, to exhibit the event $X\ge s$, it is sufficient to reveal a scattered $H_k(Q,\pi)^{(\ell)}$-matching
of size~$s$. Since every edge in such a matching can be encoded by specifying $r$ values of $\Pi_{N}$, it suffices to reveal only $rs$ values of $\Pi_N$.
Hence,  Theorem \ref{tala}, applied with $N=rn$, $c=2$, $d=r$ and $\eps=1/2$, implies that
\begin{equation}\label{TalaX}
\PP(|X-m|\ge m/2)\le4\exp(-m/(512r)).
\end{equation}

To finish the proof and show sharp concentration of $X$ around $\E X$ we need to switch from  median $m$ to  expectation $\mu:=\E X$. Since, under concentration~\eqref{TalaX}, $|m -\mu|=O(\sqrt m)$ (see for example \cite{Talagrand} or Lemma 4.6 in~\cite{McDiarmid1998}) and, by~\eqref{EXbound}, $\mu\to\infty$, it follows that $m\to\infty$ too and, in particular,  $|m-\mu|\le (1/12)\mu$. This and ~\eqref{TalaX}  imply together that $\PP(|X-m|\ge m/2) = o(1)$ and, consequently,
	\begin{align*}
		&\PP(|X-\mu|\ge(3/4)\mu)
		= \PP(|X-m + m-\mu|\ge(3/4)\mu)\\
		&\le \PP(|X-m| + |m-\mu|\ge(3/4)\mu)\le \PP(|X-m| + (1/12)\mu \ge(3/4)\mu)\\
		&= \PP(|X-m| \ge(2/3)\mu)\le \PP(|X-m|\ge m/2) = o(1).
	\end{align*}
	Hence, a.a.s.,~$X \ge (1/4)\mu=(1/4)\E X$, which combined with~\eqref{EXbound} finishes the proof of Theorem~\ref{fullhouse}.
\end{proof}

\section{Open problems}

\subsection{Orders of magnitude} There are several sets of $r$-patterns $\cP$ for which the order of magnitude of the size of the largest $\cP$-clique remains unknown. For instance, we wonder if
one can match the upper bound in Proposition~\ref{fullhouse-1} with a lower bound of the same order of magnitude. We were only able to do it in the special case of $r$-partite patterns (cf. Proposition~\ref{thm:random_three_patterns_stronger}).

An intriguing question pops up in the context of Theorem~\ref{fullhouse}. Recall that a $t$-cube $\cC(P,\pi)$ has size $2^t$ for some $t\ge1$ and, a.a.s., $z_{\cC(P,\pi)}(\rm^{(r)}_{n})=\Theta(n^{1/(r-t)})$. We believe that, for every $r$ and $t\leq r-1$, if $\cP$ is a set of $r$-patterns of size $|\cP|=2^t$, then, a.a.s.,
\[
z_{\cP}(\rm^{(r)}_{n})=O(n^{1/(r-t)})
\]
and, moreover, $z_{\cP}(\rm^{(r)}_{n})=\Theta(n^{1/(r-t)})$ only if $\cP$ is a $t$-cube.

 Ultimately, one would like to learn the spectrum of all functions $f(n)$ for which there exist sets of $r$-patterns $\cP$ such that, a.a.s.,
$z_{\cP}(\rm^{(r)}_{n})=\Theta(f(n))$. So far, all such functions have been constrained to the set
$\{n^{1/m},n^{1-1/m}: m=1,\dots,r\}.$ For instance, it is $\{1,n^{1/3}, n^{1/2}, n^{2/3}, n\}$ for $r=3$ and $\{1,n^{1/4}, n^{1/3}, n^{1/2}, n^{2/3}, n^{3/4}, n\}$ for $r=4$.
Does the spectrum extend beyond that?

\subsection{Triplets of patterns} In view of Theorem~\ref{thm:random_two_patterns_gen}, a next realistic goal might be to find  orders of magnitude of $z_{\cP}(\rm^{(r)}_{n})$ for all sets $\cP$ of triplets of $r$-patterns. Again, the most challenging case seems to be when at least two of these patterns are harmonious. So far, we only know that, a.a.s.,  $z_{\cP}(\rm^{(3)}_{n})=\Theta(n^{2/3})$ for any triplet of 3-partite 3-patterns (see  Corollary~\ref{friends}), and $z_{\cP}(\rm^{(3)}_{n})=\Theta(n^{1/2})$ for $\cP$ consisting of two 3-partite 3-patterns and the line $P_1$, as well as for $\cP=\{P_i,P_3,P_5\}$, $i=2,4$ (see \cite[Corollary 1.4]{Mgr}).

In contrast, the case when all three patterns in $\cP$ are pairwise non-harmonious  and at least one is collectable, should be more tractable. (If all three are non-collectable, we know, by the argument opening Section~\ref{non-collect}, that  $z_{\cP}(\rm^{(3)}_{n})=\Theta(1)$.)
We conjecture that then, a.a.s., $z_{\mathcal P}(\rm^{(r)}_{n})=\Theta(n^{1/r})$, the same as  for collectable singletons and mismatch pairs. We have confirmed  this only for  triplets of Dyck $r$-patterns (see Proposition~\ref{1-over-r}) and for $\cP=\{P_2,P_4,P_9\}$ (see \cite[Corollary 1.4]{Mgr}).

One obstacle in proving the conjecture is that an analog of Observation~\ref{obs:1} holds only for $r\ge5$. Indeed, it fails for a number of triplets of 4-patterns, e.g., for $$\{ABABABAB,ABBAAABB,AABBBBAA\},$$ and is doomed to fail for \emph{all} triplets of 3-patterns, since there are only two different compositions for 2-patterns (2 and $1+1$), so after removing a letter, there must be a harmonious pair.

Thus, already for $r=3$ we need either a new idea or try some ad hoc techniques for particular triplets. For instance, for $\cP=\{P_2,P_4,P_9\}$  we were able to show reconstructability of $\cP$-cliques by directly analyzing the structure of such cliques; this, further led to compute $a_{\cP}(k)$ for this $\cP$ precisely (see~\cite[Proposition 1.6]{Mgr}). On the other hand, for $\{P_3,P_5,P_7\}$
 even reconstructability fails, as the following counterexample shows: $M=ABBCAACBC$ and $N=ABACCABBC$ are both $\{P_3,P_5,P_7\}$-cliques with the same trace $\tr(M)=\tr(N)=112123233$. Recall, however, that as remarked at the end of Section~\ref{comtool}, reconstructability was not vital for the proof of Theorem~\ref{thm:random_two_patterns_gen} for mismatched pairs.

 \subsection{The role of non-collectable patterns}\label{7.3}
 We already know that, even deterministically, the largest $\cP$-cliques for sets $\cP$ consisting entirely of non-collectable patterns have size $\Theta(1)$. It is plausible that they are likewise meaningless in combination with collectable patterns. In particular, we speculate that  for every set $\mathcal P$ of collectable $r$-patterns and every set $\mathcal Q$ of non-collectable $r$-patterns, a.a.s.,
\begin{equation}\label{PQ}
z_{\mathcal P}(\rm^{(r)}_{n})=\Theta(z_{\mathcal P\cup\mathcal Q}(\rm^{(r)}_{n})).
\end{equation}

 We know it is true when $|\cP|=|\cQ|=1$ (cf. Theorem~\ref{thm:random_two_patterns_gen}, the mismatch case) as well as when $\cP=\cR^{(r)}$, the set of all $r$-partite $r$-patterns, since then, a.a.s., $z_{\mathcal P}(\rm^{(r)}_{n})=\Theta(n)$ (cf. Theorem~\ref{thm:random_r-partite}). Two other instances confirming this speculation are provided by Corollary~\ref{friends} and Proposition~\ref{1-over-r}. Indeed, in Corollary~\ref{friends} the order of magnitude of $z_{\mathcal P}(\rm^{(r)}_{n})$ remains the same for any set $\cP$ satisfying the inclusions. In particular, one can take $\cP=\cR^{(r)}\setminus\{P\}$, $P\in\cR^{(r)}$, and for $\cQ$ - all non-collectable $r$-patterns. Then,~\eqref{PQ} holds.

 As for Proposition~\ref{1-over-r}, note that the set $\cP^{(r)}_{Dyck}$ contains  exactly $2^{r-1}$  collectable patterns, one for each composition of $r$, and thus $C_r-2^{r-1}$ non-collectable ones. So, denoting these two sets by $\cP$ and $\cQ$,  ~\eqref{PQ} holds again.
In line with the above speculation, we believe that adding even all remaining $\left(\tfrac12-\tfrac1{r+1}\right)\binom{2r}r-3^{r-1}+2^{r-1}$ non-collectable patterns to $\cP^{(r)}_{Dyck}$ would not make a change. (While, in contrast, adding even a single collectable pattern $P$ to $\cP^{(r)}_{Dyck}$, in view of Theorem~\ref{thm:random_two_patterns_gen} for harmonious pairs, would increase the size of the largest clique to $\Omega(n^{1/(r-1)})$.)

For instance, for $r=4$, out of all eight non-collectable patterns only $P'=ABBBABAA$ and $P''=AABABBBA$ are not Dyck words, and we conjecture that the extended set $\cP^{(4)}_{Dyck}\cup\{P',P''\}$ should also satisfy the conclusion of Proposition~\ref{1-over-r}.

 In the context of Theorem~\ref{thm:random_r-partite},  a related question emerges:  does $z_{\cR^{(r)}}(\rm^{(r)}_{n})$ increase if one replaces $\cR^{(r)}$ with the set of all collectable $r$-patterns.
 More specifically, a matching is deemed \emph{clean} if every pair of edges forms a collectable pattern. Estimate the largest clean sub-matching of   $\rm^{(r)}_n$, that is, find a constant $c_r$ such that, a.a.s., the size of a largest clean sub-matching in $\rm^{(r)}_n$ is $\sim c_rn$. By Theorem~\ref{thm:random_r-partite} we know that $c\ge(r-1)!/r^{r-1}$.
 (The deterministic version of this question seems also interesting and is probably harder.)

\subsection{Matching processes}

There is yet another way of generating $\rm^{(r)}_n$, besides the uniform sampling and the permutational scheme described in Section~\ref{section:random}.
  In the \emph{online scheme}, for an ordered set of $rn$ vertices,  sequentially, for each  next available vertex $v$  a set of $r-1$ vertices is selected uniformly at random (from all still  available vertices) to form with $v$ an $r$-element edge.

  In this paper we have been interested in the final stage of the process, that is, in $\rm^{(r)}_n$.
 But ultimately, one could try to get  asymptotic properties of the sub-matching formed by the first $t$ edges of the process, $t=1,\dots,n$. Note that it is not the same as $\rm^{(r)}_t$, except for $t=n$. Indeed, even for $r=t=2$ the probability of getting $AABB$ is exactly $1/(2n-1)$ but each of $ABAB$ and $ABBA$ appears with probability $(n-1)/(2n-1)$. Hence, in contrary to $\rm^{(2)}_n$, for $n\ge 3$ not all matchings occur with the same probability.

 In relation to this paper, the goal might be, for a given set $\cP$ of $r$-patterns and  any sequence $1\le t=t(n)\le n$, to determine the order of magnitude of the size of a largest $\mathcal P$-clique in the ordered matching $t$-process, and reversely, given a function $k=k(n)$, to determine the threshold $t=t(n)$ (or even the hitting time) for the $t$-process to grow a $\mathcal P$-clique of size $k(n)$.


\end{document}